\documentclass[11pt]{article}
\usepackage{amsmath,mathrsfs}
\usepackage{amssymb}
\usepackage{amsthm}
\usepackage{indentfirst}
\usepackage{color}
\allowdisplaybreaks

\newtheorem{theorem}{Theorem}[section]
\newtheorem*{theorem*}{Theorem}
\newtheorem{corollary}[theorem]{Corollary}
\newtheorem{remark}[theorem]{Remark}
\newtheorem{lemma}[theorem]{Lemma}
\newtheorem{definition}[theorem]{Definition}

\newtheorem{proposition}[theorem]{Proposition}
\newtheorem{example}[theorem]{Example}
\newtheorem*{proposition*}{Proposition}

\newcommand{\R}{\mathbb{R}}

\newcommand{\E}{\mathbb{E}}

\newcommand{\Dxp}{D_{\bx_p}}

\newcommand{\be}{\begin{eqnarray*}}
\newcommand{\ee}{\end{eqnarray*}}
\newcommand{\ba}{\begin{align*}}
\newcommand{\bpm}{\begin{pmatrix}}
\newcommand{\epm}{\end{pmatrix}}

\newcommand{\bx}{\boldsymbol{x}}

\newcommand{\by}{\boldsymbol{y}}

\begin{document}
\title{Generalized partial-slice monogenic functions}

\author{Zhenghua Xu$^1$ \thanks{This work was partially supported by  the Anhui Provincial Natural Science Foundation (No. 2308085MA04) and the National Natural Science Foundation of China (No. 11801125).} Irene Sabadini$^2$ \thanks{This work was partially supported by PRIN 2022 {\em Real and Complex Manifolds: Geometry and Holomorphic Dynamics}.}\\
\emph{$^1$
\small School of Mathematics, Hefei University of Technology,}
\emph{\small  Hefei, 230601, P.R. China}\\
\emph{\small E-mail address: zhxu@hfut.edu.cn}
\\
\emph{\small $^2$ Politecnico di Milano, Dipartimento di Matematica,}
\emph{\small Via E. Bonardi, 9, 20133 Milano Italy} \\
\emph{\small E-mail address:   irene.sabadini@polimi.it}
}

\date{}

\maketitle

\begin{abstract}
The two function theories of monogenic and of slice monogenic functions have been extensively studied in the literature and were developed independently; the relations between them, e.g. via Fueter mapping and Radon transform, have been studied. The main purpose of this article is to describe a new function theory which includes   both of them as special cases.  This theory allows to prove nice properties  such as the identity theorem,  a Representation Formula, the  Cauchy (and Cauchy-Pompeiu) integral formula, the  maximum modulus principle, a version of the Taylor series  and Laurent series expansions. As a complement, we shall also offer two approaches to these functions via generalized partial-slice  functions and via global differential operators. In addition, we discuss the conformal invariance property  under a proper  group of M\"{o}bius transformations preserving the partial  symmetry of the involved domains.
\end{abstract}
{\bf Keywords:}\quad Functions of a hypercomplex variable;  monogenic functions; slice monogenic functions; Clifford algebras\\
{\bf MSC (2020):}\quad  Primary: 30G35;  Secondary: 32A30, 32A26
\section{Introduction}
As a generalization of holomorphic functions of one complex variable, the classical Clifford analysis is a function theory studying nullsolutions of the generalized Cauchy-Riemann systems such as the  solutions of the Weyl or Dirac systems, called  monogenic functions, which are  defined on domains of the Euclidean space $\mathbb{R}^{n+1}$ or $\mathbb{R}^{n}$ and with  values in the real Clifford algebra $\mathbb{R}_{n}$.  Quaternions correspond to the case  $\mathbb{R}_{2}$ and are often denoted by $\mathbb{H}$.  The study of the class of quaternionic valued functions over $\mathbb H$  and in the kernel of the so-called Cauchy-Fueter operator is somewhat special and goes back to the thirties of the past century with the works of Fueter and Moisil, \cite{Fueter,Moisil}. More in general, the interest in theories of functions in the kernel of a generalized Cauchy-Riemann operator generalizing the complex holomorphic functions started after the study
of hypercomplex algebras in the classical works of Gauss, Hamilton, Hankel, Frobenius and
goes back to the end of the nineteenth century. For a list of references, the reader may consult the early paper \cite{ward} or the book \cite{Colombo-Sabadini-Struppa-20} for an updated account.
The literature on monogenic functions is very rich and standard reference books are \cite{Brackx,Colombo-Sabadini-Sommen-Struppa-04,Gilbert,Gurlebeck}.

Compared with complex analysis, there is one major issue in the theory of monogenic functions, namely the fact  that    polynomials  and series, even    powers of   the paravector variable, fail to be  monogenic  in Clifford analysis. This fact does not prevent to obtain a series expansion in terms of suitable homogeneneous monogenic polynomials but may prevent some other features like the definition of a counterpart of exponential function keeping the standard properties, some applications to operator theory, etc. In order to study the function class including  the powers of paravectors variables,  there have been various approaches.

In 1992, Leutwiler noticed that the power functions $x^{k} (k\in \mathbb{N})$ of the  paravector  variable $x\in \mathbb{R}^{n+1}$   are   conjugate gradients of   hyperbolic harmonic functions  and started  to  investigate  solutions of   $H$-systems \cite{Leutwiler}.  This type  of functions was originally restricted to be paravector-valued and further developed into the so-called theory of  hypermonogenic functions. These functions are defined in terms of a modified Dirac operator  introduced by Eriksson and Leutwiler in 2000 \cite{Eriksson}.

In 1998,  Laville and Ramadanoff \cite{Laville-Ramadanoff} introduced the holomorphic Cliffordian functions $f:  \mathbb{R}^{2m+2} \rightarrow\mathbb{R}_{2m+1}$, which are solutions of the differential equation $D\Delta^{m} f=0$, where  $D$ and $\Delta$ are the Dirac operator and   Laplacian in $\mathbb{R}^{2m+2}$, respectively.

In 2006, the theory of slice regular functions of one quaternionic variable was initiated by Gentili and Struppa \cite{Gentili-Struppa-06,Gentili-Struppa-07}, inspired by an idea of Cullen. This function theory includes power series of the form
 \begin{equation*}\label{slice-regular-H}
f(q)=\sum\limits_{n=0}^{\infty}q^na_n, \quad   q,  a_n\in \mathbb{H}.
\end{equation*}

In 2009, Colombo, Sabadini and Struppa \cite{Colombo-Sabadini-Struppa-09} generalized  the idea in \cite{Gentili-Struppa-06,Gentili-Struppa-07}  to functions defined   on domains of the Euclidean space $\mathbb{R}^{n+1}$, identified with the set of paravectors, and with values in the Clifford algebras $\mathbb{R}_{n}$. In analogy with the classical case, they called these functions slice monogenic functions.  The slice regularity can also be studied on octonions \cite{Gentili-Struppa-10}. Based on a well-known Fueter-Sce construction, see \cite{Colombo-Sabadini-Struppa-20} for a summary, these functions have eventually been  generalized to real alternative  algebras by Ghiloni and Perotti \cite{Ghiloni-Perotti-11}.    The slice regularity has been further developed to slice Dirac-regular functions over octonions  \cite{Jin-Ren-Sabadini}.
Despite the fact that  slice monogenic (or slice regular, or slice hyperholomorphic) functions have been introduced relatively recently, they are nowadays a well developed theory especially because a strong push has been given by the applications to operator theory. The literature is rather wide and for an overview of the numerous references the reader may consult   books \cite{Alpay1,Alpay,Colombo-Gantner-20,Colombo-Gantner-Kimsey-18,Colombo-Sabadini-Struppa-11,Colombo-Sabadini-Struppa-16,Gentili-Stoppato-Struppa-13} and  references therein.

The theories of monogenic functions and of slice monogenic functions  are natural generalizations of classical complex function theory in higher dimensions,  but they are two different function theories.   In fact, the slice monogenic functions appear already in the Fueter-Sce construction which  allows to construct monogenic (or regular) functions starting from holomorphic functions.
The intermediate step in the construction gives slice monogenic functions which are intrinsic (they are also called radially holomorphic functions, see \cite{Qian-97,Dong-Qian-21} or \cite{Gurlebeck}). However, it holds also for more general slice monogenic functions, since the Laplacian (or its power) is a real operator and so it acts on the real components of a function. For a discussion on slice monogenic functions from the point of view of the Fueter-Sce construction see \cite{Colombo-Sabadini-Struppa-20} and also \cite{Colombo-Gantner-20}.   Another link between these two classes of functions is given by Radon and dual Radon transforms for Clifford algebras valued functions \cite{Colombo-Sabadini-Soucek-15}.

 In  the present paper,   we introduce a new  class of functions which includes  classical  monogenic functions and slice monogenic functions as two special cases. Compared to slice regular and/or monogenic  functions, this new  theory has also nice properties such as   identity theorem, the Representation Formula,  Cauchy   integral formula with generalized partial-slice monogenic Cauchy kernel,  maximum modulus principle,   Taylor  series and Laurent series expansion formulas.

Below we explain how this paper is organized and the main results in each section.

In Section 2,  we recall some basic definitions about
Clifford algebras and introduce the notion of  monogenic and of slice monogenic functions.

In Section 3,  we introduce the concept  of  \textit{generalized partial-slice monogenic functions}. Thereafter, we present  some examples and prove some   properties among which the analog of Splitting Lemma
(Lemma \ref{Splitting-lemma}) which says that every  generalized partial-slice monogenic function, if restricted to each slice subspace, can be split into monogenic components.  We also introduce the analogs of Fueter polynomials which, in this case, are constructed in a suitable "ad hoc" way. Lemma \ref{Splitting-lemma}  implies that some basics in  the theory of monogenic functions   hold naturally for  generalized partial-slice monogenic functions, such as the Taylor series expansion (Lemma \ref{Taylor-lemma-left}) and the Cauchy integral formula (Theorem \ref{Cauchy-formula-monogenic}).   However, those results are local.

To describe  more   global properties, the concepts of \textit{slice domain} and of \textit{partially symmetric domain} are introduced.  On these domains,  generalized partial-slice monogenic  functions satisfy  an  identity theorem (Theorem  \ref{Identity-theorem}),  which can  be  used  to  establish the Representation Formula (Theorem \ref{Representation-Formula-SM}) if the domain is, additionally,  partially symmetric.  Theorem \ref{Representation-Formula-SM} allows to prove  the Cauchy   integral formula with generalized
partial-slice  Cauchy kernel (Theorem \ref{Cauchy-formula-slice}), a global  Taylor series  (Theorem  \ref{Taylor-theorem}) and a  Laurent series expansions  (Theorem  \ref{Laurent}).  It is important to note that the Representation Formula leads to an extension result of monogenic functions on  a slice subspace but  we also introduce  a version of the Cauchy-Kovalevskaya extension starting from a real analytic function in $\mathbb R^{p+1}$. Finally, we prove a  maximum modulus principle for generalized partial-slice monogenic functions defined over  slice domains.

In  Section  4, inspired by the Fueter-Sce construction and \cite{Ghiloni-Perotti-11},  we  introduce the so-called \textit{generalized partial-slice function}. These functions naturally satisfy a representation formula (Theorem  \ref{Representation-Formula-SR}), which implies the validity of  a Cauchy-Pompeiu  integral  formula (Theorem \ref{Cauchy-Pompeiu}) for  generalized partial-slice  $C^{1}$ functions.  We shall say that a generalized partial-slice  function $f$ is called \textit{regular} if   satisfies a suitable generalized Cauchy-Riemann equations (see Definition  \ref{definition-GSR}). In Theorem  \ref{relation-GSR-GSM},  we  present a relationship between these two classes of generalized partial-slice monogenic and regular introduced in Sections 3 and 4.

In  Section 5, two  nonconstant coefficients differential operators $G_{\bx}$ and  $\overline{\vartheta}$
 are introduced to study generalized partial-slice monogenic functions.
 In Theorem  \ref{relationship-g},  we  obtain  a relationship between kernel spaces $ker G_{\bx}$, $ker \overline{\vartheta}$, and the  function class of generalized partial-slice monogenic (or regular)  functions  over partially symmetric domains.

 Finally, in  Section  6,  we investigate the conformal invariance property of generalized partial-slice monogenic functions. To this end, we define a subgroup of M\"obius group preserving the partial symmetry and describe the conformal invariance property  of generalized partial-slice monogenic functions  under this subgroup in Theorem \ref{invariance}.

\section{Preliminaries}
In this section, we collect some preliminary results on Clifford algebras and on the function theories of monogenic  functions and  of slice   monogenic  functions.
\subsection{Clifford algebras}
Let $\{e_1,e_2,\ldots, e_n\}$ be a standard orthonormal basis for the $n$-dimensional real Euclidean space  $\mathbb{R}^n$. A real Clifford algebra, denoted by $\mathbb{R}_{n}$, is generated  by these basis elements assuming that they satisfy  the relations $$e_i e_j + e_j e_i= -2\delta_{ij}, \quad 1\leq i,j\leq n,$$
where $\delta_{ij}$ is the Kronecker symbol. Hence, an arbitrary element in the Clifford algebra $\mathbb{R}_{n}$ can be written as
 $$a=\sum_A a_Ae_A, \quad a_A\in \mathbb{R},$$
 where
$$e_A=e_{j_1}e_{j_2}\cdots e_{j_r},$$
and $A=\{j_1,j_2, \cdots, j_r\}\subseteq\{1, 2, \cdots, n\}$ and $1\leq j_1< j_2 < \cdots < j_r \leq n,$  $ e_\emptyset=e_0=1$.

The norm of $a$ is defined by $|a|= ({\sum_{A}|a_{A}|^{2}} )^{\frac{1}{2}}.$ As a real  vector space, the dimension of Clifford algebra $\mathbb{R}_n$  is $2^{n}$.

For $k=0,1,\ldots,n$, the real linear subspace $\mathbb{R}_n^k$ of $\mathbb{R}_n$, called $k$-vectors, is  generated by the $\begin{pmatrix} n\\k\end{pmatrix}$ elements of the form
 $$e_A=e_{i_1}e_{i_2}\cdots e_{i_k},\quad 1\leq i_1<i_2<\cdots<i_k\leq n.$$
In particular, the element $[a]_{0}=a_{\emptyset}$ is called scalar part of $a$.

In $\mathbb{R}_n$ we can define some involutions.
The Clifford conjugation of $a$ is  defined as
$$\overline{a} =\sum_Aa_A\overline{e_A},$$
 where $\overline{e_{j_1}\ldots e_{j_r}}=\overline{e_{j_r}}\ldots\overline{e_{j_1}},\ \overline{e_j}=-e_j,1\leq j\leq n,\ \overline{e_0}=e_0=1$.
The reversion of $a$ is defined as
$$\tilde{a}=\sum_{A}a_A\widetilde{e_A},$$
where $\widetilde{e_{j_1}\cdots e_{j_r}}=e_{j_r}\cdots e_{j_1}$. Also $\widetilde{ab}=\tilde{b}\tilde{a}$ for $a, b\in\mathbb{R}_n.$

We note that an important subset of Clifford numbers $\mathbb{R}_n$ is the set of the so-called paravectors which consists of elements in $\mathbb{R}_n^{0} \oplus \mathbb{R}_n^{1}$. This subset will  be identified with $\mathbb{R}^{n+1}$ via the map
$$(x_0,x_1,\ldots,x_n) \longmapsto   x=x_{0}+\underline{x}=\sum_{i=0}^{n}e_ix_i.$$
For a paravector $x\neq0$, its norm is given by $|x|=(x\overline x)^{1/2}$  and so its inverse is given by $x^{-1}=  \overline{x}|x|^{-2}.$

\subsection{M\"obius transformations}

Since they will be used later on, we revise some basic notions on M\"obius transformations in this framework.
The  Clifford group denoted by $\Gamma_{n}$ consists of all $a \in \mathbb{R}_{n}$   which can be written as a finite product of non-zero
paravectors in $\mathbb{R}^{n+1}$. For a domain $U$ in $\R^{n+1}$, a diffeomorphism $\phi: U\longrightarrow \mathbb{R}^{n+1}$ is said to be conformal if, for each $x\in U$ and each $\mathbf{u,v}\in TU_x$, the angle between $\mathbf{u}$ and $\mathbf{v}$ is preserved under the corresponding differential $d\phi_x$ at $x$.
For $n\geq 2$, a theorem of Liouville shows that  the only conformal transformations  are M\"obius transformations. It is well known that  any M\"{o}bius transformation on $\mathbb{R}^{n+1} \cup \{\infty\}$ can be expressed as
$M:\mathbb{R}^{n+1} \cup \{\infty\} \rightarrow \mathbb{R}^{n+1} \cup \{\infty\},$
 $$M \langle x\rangle :=(ax+b)(cx+d)^{-1}, \quad M=\begin{pmatrix} a \ b \\ c\ d\end{pmatrix},$$
where $ a,  b,  c,  d\in \mathbb{R}_n$ satisfying     \textit{Ahlfors-Vahlen conditions} \cite{Ahlfors-85,Ahlfors-86}:\\
(i) $ a,  b,  c,  d \in \Gamma_{n}\cup \{0\};$\\
(ii) $a\tilde{b},  c\tilde{d}, \widetilde{c}a,\widetilde{d}b \in\mathbb{R}^{n+1};$\\
(iii) $a\tilde{d}-b\tilde{c}\in \mathbb{R}\backslash \{0\}.$

Equipped with the product of matrices, the set that consists of matrices   $\begin{pmatrix} a \ b \\ c\ d\end{pmatrix}$,
where  coefficients $a,b,c,d\in \mathbb{R}_n$ satisfy   Ahlfors-Vahlen conditions,  forms a group, called the general Ahlfors-Vahlen  group and denoted by $GAV( \mathbb{R}\oplus \mathbb{R}^{n})$.  These conditions ensure that   M\"{o}bius transformations  map from $\mathbb{R}^{n+1}$ to itself.  The group $GAV( \mathbb{R}\oplus \mathbb{R}^{n})$ can be generated by the following four types of matrices (see e.g., \cite{Elstrodt-87}):


(i) Translations: $\begin{pmatrix}1 \ \ b \\ 0 \ \ 1\end{pmatrix} $,  $ b\in \mathbb{R}^{n+1}$,
  induces the M\"obius transformation   $M\langle x\rangle =x+b$;

(ii) Inversion: $\begin{pmatrix}  0 \ \ 1\\ -1\ 0 \end{pmatrix},$
  induces $M\langle x\rangle =-x^{-1}$;

(iii) Rotations: $\begin{pmatrix}  a\   \ \ 0 \\ 0 \ \ a \end{pmatrix}$,   $ a\in \mathbb{R}^{n+1}$ with $|a|=1$,
induces  $M\langle x\rangle =ax\overline{a}$;

(iv) Dilations: $\begin{pmatrix}  \lambda\   \ \ 0 \\ 0 \ \ \lambda^{-1} \end{pmatrix},$   $\lambda \in \mathbb{R}\backslash \{0\},$
induces  $M\langle x\rangle =\lambda^{2}x$.

\subsection{Monogenic and slice monogenic functions}
In this  subsection,  we  recall the definitions of  monogenic and of slice  monogenic  functions.
\begin{definition}[Monogenic function]\label{monogenic-Clifford}
Let $\Omega $ be a domain (i.e., an open and connected  set) in $\mathbb{R}^{n+1}$ and let  $f:\Omega \rightarrow \mathbb{R}_{n}$   be   a function with continuous partial derivatives.  The function $f=\sum_{A}  e_{A}   f_{A}$ is called  (left)  monogenic in $\Omega $ if it satisfies the generalized Cauchy-Riemann equation
$$ Df(x)=\sum _{i=0}^{n}e_{i} \frac{\partial f}{\partial x_{i}}(x)
= \sum _{i=0}^{n} \sum_{A} e_{i}e_{A} \frac{\partial f_{A}}{\partial x_{i}}(x)=0, \quad x\in \Omega. $$
Similarly, the function $f$ is called  right   monogenic in $\Omega $ if
$$ f(x)D =\sum _{i=0}^{n} \frac{\partial f}{\partial x_{i}}(x)e_{i}= \sum _{i=0}^{n} \sum_{A} e_{A} e_{i} \frac{\partial f_{A}}{\partial x_{i}}(x)=0, \quad x\in \Omega. $$
\end{definition}
The operator
$$D=\sum_{i=0}^{n}e_i\frac{\partial}{\partial x_i}=\sum_{i=0}^{n}e_i\partial_{x_i}=\partial_{x_0}+\partial_{\underline{x}}$$
in the above definition is called Weyl operator  (or  generalized Cauchy-Riemann operator)  in $\mathbb{R}^{n+1}$, while  $\partial_{\underline{x}}$ denotes the classical Dirac operator in $\mathbb{R}^{n}$. Both null-solutions to  Weyl operator and  Dirac operator  are called monogenic in Clifford analysis.

 Note that all monogenic functions are real analytic. Furthermore, monogenic functions   are harmonic in view of  the equalities
$$
D\overline{D} =\overline{D} D =\Delta _{n+1}, \quad \partial_{\underline{x}} ^{2}=-\Delta _{n},$$
where $\Delta_n$ and $\Delta _{n+1}$ are the Laplacian in $\mathbb{R}^n$ and $\mathbb{R}^{n+1}$, respectively.

Note that every nonreal  paravector   in $\mathbb{R}^{n+1}$  can be written in the form
$x=x_{0}+r  \omega$ with $r= |\underline{x}|=\big(\sum_{i=1}^{n} x_i^{2}\big)^{\frac{1}{2}}$ and $\omega$, which is uniquely determined, behaves  as  a classical imaginary unit, that is
$$\omega \in S^{n-1}=\big\{x\in \mathbb{R}^{n+1}: x^2 =-1\big\}.$$
When $x$ is real, then $r=0$ and for every $\omega\in S^{n-1}$ one can write  $x=x+\omega\cdot 0$.

Based on this observation, the definition of slice  monogenic functions is introduced as follows, see \cite{Colombo-Sabadini-Struppa-09,Colombo-Sabadini-Struppa-11}.
\begin{definition}[Slice monogenic function]\label{slice-monogenic-Clifford}
 Let $\Omega$ be a domain in $\mathbb{R}^{n+1}$. A function $f :\Omega \rightarrow \mathbb{R}_{n}$ is called (left)  slice monogenic if, for all $ \omega \in   S^{n-1}$, its restriction $f_{\omega}$ to $\Omega_{\omega}=\Omega\cap (\mathbb{R}  \oplus \omega\mathbb{R})\subseteq \mathbb{R}^{2}$ is  holomorphic, i.e., it has continuous partial derivatives and satisfies
$$ (\partial_{x_0} +\omega\partial_{r}) f_{\omega}(x_0+r\omega)=0$$
for all $x_0+r\omega \in \Omega_{\omega}$.
 \end{definition}

\begin{remark} {\rm Both the theories  of monogenic and of slice monogenic functions are defined on paravectors and they both give holomorphic functions for $n=1$. The case $n=1$ is the only case in which the domain and the range are equidimensional, besides the quaternionic regular case and the octonionic case.

Specifically, in  quaternionic analysis,  Fueter  \cite{Fueter} in 1934 introduced the following  generalized Cauchy-Riemann operator, now called Cauchy-Riemann-Fueter operator,
 $$\mathcal{D}=\frac{\partial}{\partial x_{0}}+i\frac{\partial}{\partial x_{1}}+j\frac{\partial}{\partial x_{2}}+k\frac{\partial}{\partial x_{3}}$$
 where $(x_{0}, x_{1},x_{2},x_{3})$  is  the real coordinate  of the quaternion $q=x_{0}+x_{1}i+x_{2}j+x_{3}k$.
 See e.g. \cite{Colombo-Sabadini-Sommen-Struppa-04,Fueter,Gurlebeck,Sudbery-79}  for more details on  regular functions.}
\end{remark}

\section{Generalized partial-slice monogenic functions}

In the sequel, let $p$ and $q$ be a nonnegative and a positive integer, respectively. We consider functions $f:\Omega\longrightarrow  \mathbb{R}_{p+q}$ where $\Omega\subseteq\R^{p+q+1}$ is a domain.

An element $\bx\in\R^{p+q+1}$ will be identified with a paravector in $\mathbb{R}_{p
+q}$ and we shall write it as
$$\bx=\bx_p+\underline{\bx}_q \in\R^{p+1}\oplus\R^q, \quad \bx_p=\sum_{i=0}^{p}x_i e_i,\ \underline{\bx}_q=\sum_{i=p+1}^{p+q}x_i e_i.$$
This splitting of $\bx\in\R^{p+q+1}$ as $\R^{p+1}\oplus\R^q$ is fixed throughout the paper.

Similarly, the generalized Cauchy-Riemann operator   and the Euler operator are split as

\begin{equation}\label{Dxx}
D_{\bx}=\sum_{i=0}^{p+q}e_i\partial_{x_i}=\sum_{i=0}^{p}e_i\partial_{x_i}+
\sum_{i=p+1}^{p+q}e_i\partial_{x_i}=:D_{\bx_p}+D_{\underline{\bx}_q},
\end{equation}
\begin{equation*}\label{Exx}
\E_{\bx}=\sum_{i=0}^{p+q}x_i\partial_{x_i}=\sum_{i=0}^{p}x_i\partial_{x_i}+
\sum_{i=p+1}^{p+q}x_i\partial_{x_i}=:\E_{\bx_p}+\E_{\underline{\bx}_q}.
\end{equation*}

Let us denote by $\mathbb{S}$ the sphere of unit $1$-vectors in $\mathbb R^q$, whose elements $(x_{p+1},\ldots, x_{p+q})$ are identified with $\underline{\bx}_q=\sum_{i=p+1}^{p+q}x_i e_i$, i.e.
$$\mathbb{S}=\big\{\underline{\bx}_q: \underline{\bx}_q^2 =-1\big\}=\big\{\underline{\bx}_q=\sum_{i=p+1}^{p+q}x_i e_i:\sum_{i=p+1}^{p+q}x_i^{2}=1\big\}.$$
Note that, for $\underline{\bx}_q\neq0$, there exists a uniquely determined $r\in \mathbb{R}^{+}=\{x\in \mathbb{R}: x>0\}$ and $\underline{\omega}\in \mathbb{S}$, such that $\underline{\bx}_q=r\underline{\omega}$, more precisely
 $$r=|\underline{\bx}_q|, \quad \underline{\omega}=\frac{\underline{\bx}_q}{|\underline{\bx}_q|}. $$
 When $\underline{\bx}_q= 0$ we set $r=0$ and $\underline{\omega}$ is not uniquely defined, in fact for every $\underline{\omega}\in \mathbb{S}$ we have $\bx=\bx_p+\underline{\omega} \cdot 0$.

The upper half-space $\mathrm{H}_{\underline{\omega}}$ in $\mathbb{R}^{p+2}$ associated with $\underline{\omega}\in \mathbb{S}$ is defined by
$$\mathrm{H}_{\underline{\omega}}=\{\bx_p+r\underline{\omega}, \bx_p \in\R^{p+1}, r\geq0 \},$$
and it is clear from the previous discussion that
$$ \R^{p+q+1}=\bigcup_{\underline{\omega}\in \mathbb{S}} \mathrm{H}_{\underline{\omega}},$$
and
$$ \R^{p+1}=\bigcap_{\underline{\omega}\in \mathbb{S}} \mathrm{H}_{\underline{\omega}}.$$

In the sequel, we shall make use of the notation
$$
\Omega_{\underline{\omega}}:=\Omega\cap (\mathbb{R}^{p+1} \oplus \underline{\omega} \mathbb{R})\subseteq \mathbb{R}^{p+2}
$$
where $\Omega$ is an open set in $\mathbb{R}^{p+q+1}$.

Recalling the notation in formula \eqref{Dxx},  we now introduce the  definition of  generalized partial-slice monogenic functions as follows.
\begin{definition} \label{definition-slice-monogenic}
 Let $\Omega$ be a domain in $\mathbb{R}^{p+q+1}$. A function $f :\Omega \rightarrow \mathbb{R}_{p+q}$ is called left  generalized partial-slice monogenic of type $(p,q)$ if, for all $ \underline{\omega} \in \mathbb S$, its restriction $f_{\underline{\omega}}$ to $\Omega_{\underline{\omega}}\subseteq \mathbb{R}^{p+2}$  has continuous partial derivatives and  satisfies
$$D_{\underline{\omega}}f_{\underline{\omega}}(\bx):=(D_{\bx_p}+\underline{\omega}\partial_{r}) f_{\underline{\omega}}(\bx_p+r\underline{\omega})=0,$$
for all $\bx=\bx_p+r\underline{\omega} \in \Omega_{\underline{\omega}}$.
 \end{definition}
We denote by $\mathcal {GSM}(\Omega)$ (or $\mathcal {GSM}^{L}(\Omega)$ when needed) the function class of  all  left generalized partial-slice monogenic functions {of type $(p,q)$} in $\Omega$.  {It is immediate to verify that $\mathcal {GSM}^{L}(\Omega)$ is a right Clifford-module over $\mathbb R_{p+q}$.} Similarly, denote by $\mathcal {GSM}^{R}(\Omega)$ the left Clifford module of  all right  generalized partial-slice monogenic functions of type $(p,q)$ $f:\Omega  \rightarrow \mathbb{R}_{p+q}$ which are defined by requiring that the restriction $f_{\underline{\omega}}$ to $\Omega_{\underline{\omega}}$ for  $ \underline{\omega} \in \mathbb S$ and satisfies
$$f_{\underline{\omega}}(\bx)D_{\underline{\omega}}:={f_{\underline{\omega}} (\bx_p+r\underline{\omega})D_{\bx_p}}+ \partial_{r}f_{\underline{\omega}} (\bx_p+r\underline{\omega})\underline{\omega}=0.$$
\begin{remark}{\rm To ease the terminology, in the sequel we shall omit to specify the type $(p,q)$ since it will be clear from the context.}
\end{remark}
\begin{remark}\label{rem32}
{\rm
 When $(p,q)=(n-1,1)$, the  notion of generalized partial-slice monogenic functions in Definition  \ref{definition-slice-monogenic} coincides with the notion of  classical monogenic functions  defined in $\Omega\subseteq\mathbb{R}^{n+1}$   with values in the Clifford algebra $\mathbb{R}_{n}$, which is denoted by $\mathcal {M}(\Omega)$. See Definition  \ref{monogenic-Clifford}.

When $(p,q)=(0,n)$, Definition  \ref{definition-slice-monogenic}  boils down to that one of  slice monogenic functions in Definition \ref{slice-monogenic-Clifford} defined in $\mathbb{R}^{n+1}$ and   with values in the Clifford algebra $\mathbb{R}_{n}$, see \cite{Colombo-Sabadini-Struppa-09}.}
\end{remark}

 Before discussing properties of the function class $\mathcal {GSM}(\Omega)$,  we  provide some examples to illustrate the variety of functions belonging to this class.
\begin{example}\label{example-1}
With the above notations, let $\underline{\omega} \in\mathbb{S}$, $D=\mathrm{H}_{\underline{\omega}}\cup \mathrm{H}_{-\underline{\omega}}$ and set
\begin{eqnarray*}
f(\bx)=
\left\{
\begin{array}{ll}
0,     &\mathrm {if} \ \bx \in  \mathbb{R}^{p+q+1} \setminus D,
\\
1,   &\mathrm {if} \ \bx \in D \setminus \mathbb{R}^{p+1}.
\end{array}
\right.
\end{eqnarray*}
 Then $f \in \mathcal {GSM}^{L}(\mathbb{R}^{p+q+1} \setminus \mathbb{R}^{p+1}) \cap{\mathcal{GSM}}^{R}(\mathbb{R}^{p+q+1} \setminus \mathbb{R}^{p+1}).$
\end{example}

\begin{example}Given $n\in \mathbb{N}$, set
$$f(\bx)=(x_{0}+\underline{\bx}_{q})^{n}, \quad \bx=  \sum_{i=0}^{p}x_i e_i + \underline{\bx}_q, \ \underline{\bx}_q=\sum_{i=p+1}^{p+q}x_i e_i.$$
Then $f \in \mathcal {GSM}^{L}(\mathbb{R}^{p+q+1}) \cap{\mathcal{GSM}}^{R}(\mathbb{R}^{p+q+1}).$
\end{example}
\begin{example}\label{Cauchy-kernel-example}
Let $\Omega=\mathbb{R}^{p+q+1}\setminus \{0\}$ and set
$$E(\bx):=\frac{1}{\sigma_{p+1}}\frac{\overline{\bx}}{|\bx|^{p+2}}, \quad \bx \in\Omega,$$
where $\sigma_{p+1}=2\frac{\Gamma^{p+2}(\frac{1}{2}) }{\Gamma (\frac{p+2}{2})} $ is the surface area of the unit ball in $\mathbb{R}^{p+2}$. \\
Then $E(\bx) \in \mathcal {GSM}^{L}(\Omega) \cap{\mathcal{GSM}}^{R}(\Omega).$
\end{example}

\begin{remark}{\rm
As it is well-known, the Cauchy kernel  for monogenic functions defined on paravectors $\mathbb{R}^{p+1}$ is
$${\mathsf E}(\bx)=\frac{1}{\sigma_{p}}\frac{\overline{\bx}}{|\bx|^{p+1}}, \quad \bx=\sum_{i=0}^{p}x_i e_i\in \mathbb{R}^{p+1}\setminus \{0\}.$$
Compared with the classical  monogenic Cauchy kernel,   the remarkable point in Example  \ref{Cauchy-kernel-example} is the freedom of the dimension $q$.
In other words, the  new function class considered in this paper not only includes the two classes of  classical monogenic and of slice monogenic functions, but it is much larger than their synthesis.}
\end{remark}


Next result extends the Splitting Lemma, see \cite{Colombo-Sabadini-Struppa-09}, to the present case:

\begin{lemma} {\bf(Splitting lemma)}\label{Splitting-lemma}
Let $\Omega\subseteq \mathbb{R}^{p+q+1}$ be a domain and $f:\Omega\rightarrow \mathbb{R}_{p+q}$ be a  generalized partial-slice monogenic function. For
every arbitrary, but fixed $\underline{\omega} \in\mathbb{S},$ also denoted by $I_{p+1}$, let $I_{1},I_{2},\ldots,I_{p},I_{p+2},\ldots,I_{p+q}$ be a completion to a basis of $\mathbb{R}_{p+q}$ such that $I_{r}I_{s}+I_{s}I_{r}=-2\delta_{rs}$, $r,s=1,\ldots, p+q$, $I_r\in\mathbb{R}^p$, $r=1,\ldots, p$, $I_r\in\mathbb{R}^q$, $r=p+2,\ldots,p+q$. Then there exist $2^{q-1}$ monogenic functions $F_{A}:\Omega_{\underline{\omega}} \rightarrow \mathbb{R}_{p+1}={\rm Alg}\{I_1,\ldots, I_p, I_{p+1}=\underline{\omega}\}$  such that
$$ f_{\underline{\omega}}(\bx_p+r\underline{\omega})=\sum_{A=\{i_{1},\ldots, i_{s}\} \subseteq \{p+2,\ldots,p+q\}}F_{A}(\bx_p+r\underline{\omega})I_{A}, \quad \bx_p+r\underline{\omega}\in \Omega_{\underline{\omega}},$$
where $I_{A}=I_{i_{1}}\cdots I_{i_{s}},  A=\{i_{1},\ldots, i_{s}\} \subseteq \{p+2,\ldots,p+q\}$ with $i_{1}<\cdots<i_{s}$, and $I_{\emptyset}=1$ when $A=\emptyset$.
\end{lemma}
\begin{proof}
Any function $f:\ \Omega\subseteq \mathbb{R}^{p+q+1}\to\mathbb{R}_{p+q}$ can be written as
$$
f(\bx)=\sum_{A=\{i_{1},\ldots, i_{s}\} \subseteq \{p+2,\ldots,p+q\}}F_{A}(\bx)I_{A}, \quad
$$
where the functions $F_A$ have values in $\mathbb{R}_{p+1}={\rm Alg}\{I_1,\ldots, I_p, I_{p+1}=\underline{\omega}\}$. Let us now consider the restriction of the function $f$ to $\mathbb{R}^{p+2}\cong\mathbb{R}^{p+1}+\mathbb{R}\underline{\omega}$ so that
$$
f(\bx_p+r\underline{\omega})=\sum_{A=\{i_{1},\ldots, i_{s}\} \subseteq \{p+2,\ldots,p+q\}}F_{A}(\bx_p+r\underline{\omega})I_{A}, \quad \bx_p+r\underline{\omega} \in\Omega_{\underline{\omega}}.
$$
Since $f_{\underline{\omega}}$ satisfies $D_{\underline{\omega}}f_{\underline{\omega}}=0$, we have
$$
\sum_{A=\{i_{1},\ldots, i_{s}\} \subseteq \{p+2,\ldots,p+q\}}(D_{\bx_p}+\underline{\omega}\partial_{r}) F_{A}(\bx_p+r\underline{\omega})I_{A}=0,
$$
and since $(D_{\bx_p}+\underline{\omega}\partial_{r}) F_{A}(\bx_p+r\underline{\omega})$ takes values in $\mathbb{R}_{p+1}={\rm Alg}\{I_1,\ldots, I_p, I_{p+1}=\underline{\omega}\}$, by the linear independence of the elements of the basis of $\mathbb{R}_{p+q}$ considered as a real linear space, we deduce that
$(D_{\bx_p}+\underline{\omega}\partial_{r}) F_{A}(\bx_p+r\underline{\omega})=0$, namely the functions $F_A$ are monogenic on $\Omega_{\underline{\omega}}$.
\end{proof}

\begin{remark}{\rm
We stated our result in generality, but in most cases it is sufficient to take as basis of $\mathbb R^p$ the standard basis $I_r=e_r$ for $r=1,\ldots, p$.}
\end{remark}
From  Example \ref{example-1}, we deduce that the class of generalized partial-slice monogenic functions  is so large that may contain discontinuous functions. Therefore, to obtain some results it is necessary to add appropriate hypotheses on their domain. Mimicking the situation in the slice monogenic context, see \cite{Colombo-Sabadini-Struppa-11}, we give the next definition:
\begin{definition} \label{slice-domain}
 Let $\Omega$ be a domain in $\mathbb{R}^{p+q+1}$.

1.   $\Omega$ is called  slice domain if $\Omega\cap\mathbb R^{p+1}\neq\emptyset$  and $\Omega_{\underline{\omega}}$ is a domain in $\mathbb{R}^{p+2}$ for every  $\underline{\omega}\in \mathbb{S}$.

2.   $\Omega$   is called  partially  symmetric with respect to $\mathbb R^{p+1}$ (p-symmetric for short) if, for   $\bx_{p}\in\R^{p+1}, r \in \mathbb R^{+},$ and $ \underline{\omega}  \in \mathbb S$,
$$\bx=\bx_p+r\underline{\omega} \in \Omega\Longrightarrow [\bx]:=\bx_p+r \mathbb S=\{\bx_p+r \underline{\omega}, \ \  \underline{\omega}\in \mathbb S\} \subseteq \Omega. $$
 \end{definition}

Denote by $\mathcal{Z}_{f}(\Omega)$  the zero set of the function $f:\Omega\subseteq\mathbb{R}^{n+1}\rightarrow \mathbb{R}_{n}$.

Let us recall an identity theorem for monogenic functions; see e.g.,  \cite[Theorem 9.27]{Gurlebeck}.

\begin{theorem}\label{Identity-theorem-monogenic}
Let $\Omega\subseteq \mathbb{R}^{n+1}$ be a  domain and $f:\Omega\rightarrow \mathbb{R}_{n}$ be a  monogenic function.
If $\mathcal{Z}_{f}(\Omega)$ contains  a   $n$-dimensional smooth manifold, then $f\equiv0$ in  $\Omega$.
\end{theorem}

Using this result we can prove an identity theorem for generalized partial-slice monogenic functions over slice domains.
\begin{theorem} \label{Identity-lemma}
Let $\Omega\subseteq \mathbb{R}^{p+q+1}$ be a  slice domain and $f:\ \Omega\rightarrow \mathbb{R}_{p+q}$ be a  generalized partial-slice monogenic function.
If there is an imaginary $\underline{\omega}  \in \mathbb S $ such that $\mathcal{Z}_{f}(\Omega) \cap \Omega_{\underline{\omega}}$ is a   $(p+1)$-dimensional smooth manifold, then $f\equiv0$ in  $\Omega$.
\end{theorem}
\begin{proof}
Let $f$ be a generalized partial-slice monogenic function on the slice domain $\Omega$.
By Lemma \ref{Splitting-lemma},  there exist $2^{q-1}$ monogenic functions $F_{A}:\Omega_{\underline{\omega}}\subseteq\mathbb{R}^{p+2} \rightarrow \mathbb{R}_{p+1}$  such that
\begin{equation*}\label{Splitting-lemma-indentity}
f_{\underline{\omega}}(\bx_p+r\underline{\omega})=\sum_{A=\{i_{1},\ldots, i_{s}\} \subseteq \{p+2,\ldots,p+q\}}F_{A}(\bx_p+r\underline{\omega})I_{A}, \quad \bx_p+r\underline{\omega}\in \Omega_{\underline{\omega}}.
\end{equation*}
Under the hypotheses that the zero set of $f$ in the domain $\Omega_{\underline{\omega}}$ is a   $(p+1)$-dimensional smooth manifold, from the Splitting Lemma  we infer that also all the functions $F_{A}$ vanish on such a manifold  for every multi-index $A$. Thus, since the $F_A$ are monogenic, they vanish identically in  $\Omega_{\underline{\omega}}$  from Theorem  \ref{Identity-theorem-monogenic}. In particular, $F_{A}$ vanishes  in a  domain $U\subseteq\Omega_{\underline{\omega}}\cap \mathbb{R}^{p+1}$ for  every multi-index $A$, and so also $f_{\underline{\omega}}$ vanishes on $U$, namely $f$ vanishes on $U$. For any other  $\underline{\eta}\in \mathbb{S}$,   we consider the   restriction $f_{\underline{\eta}}$ of $f$ to $\Omega_{\underline{\eta}}$. By  using again Lemma \ref{Splitting-lemma},  we have
$$f_{\underline{\eta}}(\bx_p+r\underline{\eta})=\sum_{B=\{i_{1},\ldots, i_{s}\} \subseteq \{p+2,\ldots,p+q\}}G_{B}(\bx_p+r\underline{\eta})I_{B},
 \quad \bx_p+r\underline{\eta}\in \Omega_{\underline{\eta}}.$$
 Since $f$ vanishes on $U$ also $f_{\underline{\eta}}$ does. This implies that $G_{B}$ vanishes in the domain $U\subseteq \mathbb{R}^{p+1}$ for  every multi-index $B$, and so it vanishes identically in the domain $\Omega_{\underline{\eta}}$  by Theorem \ref{Identity-theorem-monogenic}. Since  $f_{\underline{\eta}}\equiv0$  in $\Omega_{\underline{\eta}}$ for all $\underline{\eta} \in \mathbb{S}$ we deduce that  $f\equiv0$ in  $\Omega$, as desired.
\end{proof}

Theorem \ref{Identity-lemma} implies directly the following result.
\begin{theorem}  {\bf(Identity theorem)}\label{Identity-theorem}
Let $\Omega\subseteq \mathbb{R}^{p+q+1}$ be a  slice domain and $f,g:\Omega\rightarrow \mathbb{R}_{p+q}$ be   generalized partial-slice monogenic functions.
If there is an imaginary $\underline{\omega}  \in \mathbb S $ such that $f=g$ on a   $(p+1)$-dimensional smooth manifold in $\Omega_{\underline{\omega}}$, then $f\equiv g$ in  $\Omega$.
\end{theorem}

The identity theorem for generalized partial-slice monogenic  functions allows to establish a representation formula.
\begin{theorem}  {\bf(Representation Formula, I)}  \label{Representation-Formula-SM}
Let $\Omega\subseteq \mathbb{R}^{p+q+1}$ be a p-symmetric slice domain and $f:\Omega\rightarrow \mathbb{R}_{p+q}$ be a  generalized partial-slice monogenic function.  Then, for any $\underline{\omega}\in \mathbb{S}$ and for $\bx_p+r\underline{\omega} \in \Omega$,
\begin{equation}\label{Representation-Formula-eq}
f(\bx_p+r \underline{\omega})=\frac{1}{2} (f(\bx_p+r\underline{\eta} )+f(\bx_p-r\underline{\eta}) )+
\frac{ 1}{2} \underline{\omega}\underline{\eta} (  f(\bx_p-r\underline{\eta} )-f(\bx_p+r\underline{\eta})),
\end{equation}
for any $\underline{\eta}\in \mathbb{S}$.

Moreover, the following two functions do not depend on $\underline{\eta}$:
$$F_1(\bx_p,r)=\frac{1}{2} (f(\bx_p+r\underline{\eta} )+f(\bx_p-r\underline{\eta} ) ),$$
$$F_2(\bx_p,r)=\frac{ 1}{2}\underline{\eta}(  f(\bx_p-r\underline{\eta} )-f(\bx_p+r\underline{\eta})).$$
\end{theorem}

\begin{proof} Consider a fixed $\underline{\eta}\in \mathbb{S}$ and the function defined by
\begin{eqnarray*}
 h(\bx)
&=&\frac{1}{2} (f(\bx_p+r\underline{\eta})+f(\bx_p-r\underline{\eta}) )+
\frac{ 1}{2} \underline{\omega}\underline{\eta}(  f(\bx_p-r\underline{\eta})-f(\bx_p+r\underline{\eta}))
 \\
 &=&\frac{ 1}{2} (1-\underline{\omega}\underline{\eta})    f(\bx_p+r\underline{\eta})+
 \frac{ 1}{2} (1+\underline{\omega}\underline{\eta})    f(\bx_p-r\underline{\eta})
\end{eqnarray*}
for  $\bx=\bx_p+r\underline{\omega} \in \Omega$ with  $\bx_{p}\in\R^{p+1}, r \geq0,$ and  $\underline{\omega} \in \mathbb{S}$.

It is immediate that
 $f\equiv h$  in the domain  $\Omega\cap \mathbb{R}^{p+1}$. If we show that $h\in\mathcal {GSM}(\Omega)$, the result will follow from Theorem \ref{Identity-theorem}.
Recalling the notation in \eqref{Dxx}, we note that
$$D_{\bx_p} \underline{\omega}\, \underline{\eta}= \underline{\omega}\overline{D}_{\bx_p}\underline{\eta}
 =\underline{\omega}\underline{\eta}D_{\bx_p},  \quad
  \underline{\omega}\partial_{r}\underline{\omega}\underline{\eta}=-\underline{\eta}\partial_{r},$$
 so that
 $$ (D_{\bx_p}+\underline{\omega}\partial_{r}) (1-\underline{\omega}\underline{\eta})
 =(1-\underline{\omega}\underline{\eta})D_{\bx_p}+  (\underline{\omega}+\underline{\eta})\partial_{r}
 =(1-\underline{\omega}\underline{\eta})(D_{\bx_p}+\underline{\eta}\partial_{r}),$$
 and
 $$ (D_{\bx_p}+\underline{\omega}\partial_{r}) (1+\underline{\omega}\underline{\eta})
 =(1+\underline{\omega}\underline{\eta})D_{\bx_p}+  (\underline{\omega}-\underline{\eta})\partial_{r}
 =(1+\underline{\omega}\underline{\eta})(D_{\bx_p}-\underline{\eta}\partial_{r}).$$
Hence we have
 \begin{eqnarray*}
 (D_{\bx_p}+\underline{\omega}\partial_{r}) h_{\underline{\omega}}(\bx_p+r\underline{\omega})=
   \frac{1}{2}(1-\underline{\omega}\underline{\eta})   (D_{\bx_p}+\underline{\eta}\partial_{r})f(\bx_p+r\underline{\eta})
 \\
+
   \frac{1}{2}(1+\underline{\omega}\underline{\eta})    (D_{\bx_p}-\underline{\eta}\partial_{r}) f(\bx_p-r\underline{\eta}),
\end{eqnarray*}
and since  $f\in\mathcal {GSM}(\Omega)$ we immediately deduce that  $h\in\mathcal {GSM}(\Omega)$ and the formula \eqref{Representation-Formula-eq} follows.

From the validity of (\ref{Representation-Formula-eq}), we have that
\begin{equation}\label{Representation-Formula-3}
f(\bx_p-r \underline{\omega})=\frac{1}{2} (f(\bx_p-r\underline{\eta} )+f(\bx_p+r\underline{\eta} ) )+
\frac{ 1}{2} \underline{\omega}\underline{\eta}(  f(\bx_p+r\underline{\eta})-f(\bx_p-r\underline{\eta}))
\end{equation}
and combining  (\ref{Representation-Formula-eq}) with (\ref{Representation-Formula-3}),  we obtain
$$f(\bx_p+r \underline{\omega} )+f(\bx_p-r \underline{\omega}  )=f(\bx_p+r\underline{\eta})+f(\bx_p-r\underline{\eta}), $$
which means $F_1$, and similarly $F_2$, do not depend on $\underline{\eta}$.
 The proof is then complete.
\end{proof}
\begin{corollary}
The function $F_2$  given in Theorem \ref{Representation-Formula-SM} vanishes at all points of $\Omega\cap \mathbb{R}^{p+1}$.
\end{corollary}
\begin{proof}
Indeed from its definition we have
$$F_2(\bx_p,0)=\frac{ 1}{2}\underline{\eta}(  f(\bx_p)-f(\bx_p))=0.$$
\end{proof}

\begin{remark}{\rm Under the hypotheses of Theorem \ref{Representation-Formula-SM}, a more general version of the assertion holds, see Theorem \ref{Representation-Formula-SR}. In fact, \eqref{Representation-Formula-eq} implies that we can write $f(\bx_p+r \underline{\omega} )=F_1(\bx_p,r)+\underline{\omega} F_2(\bx_p,r)$ and some algebraic manipulations give the result. }
\end{remark}

As a corollary of the Representation Formula, we can prove an extension theorem that allows to construct a generalized partial-slice monogenic function starting from a function $f_{\underline{\eta}}$ that is defined on a suitable open set of the form $\Omega_{\underline{\eta}}\subseteq\mathbb{R}^{p+1}+ \underline{\eta}\mathbb{R}$ for some $\underline{\eta}\in\mathbb S$ and is in the kernel of $(D_{\bx_p}+ \underline{\eta}\partial_r)$.

\begin{theorem}[Extension theorem]\label{extthm}
 Let $\Omega\subseteq \mathbb{R}^{p+q+1}$ be a p-symmetric slice domain.
Let $f_{\underline{\eta}}:\ \Omega_{\underline{\eta}} \to\mathbb{R}_{p+q}$ be a  function with continuous partial derivatives and satisfying
$$(D_{\bx_p}+ \underline{\eta}\partial_r)f_{\underline{\eta}}(\bx_p+r\underline{\eta})=0, \qquad
\bx_p+r \underline{\eta}\in\Omega_{\underline{\eta}},
$$
for a given $\underline{\eta}\in\mathbb S$. Then, for any $\bx_p+\underline{\bx}_q=\bx_p+r\underline{\omega} \in\Omega$, the function defined by
\begin{equation*}
{\rm ext}(f_{\underline{\eta}})(\bx_p+ r\underline{\omega}):
=\frac{1}{2} (f(\bx_p+r\underline{\eta} )+f(\bx_p-r\underline{\eta}) )+
\frac{ 1}{2} \underline{\omega}\underline{\eta} (  f(\bx_p-r\underline{\eta} )-f(\bx_p+r\underline{\eta}))
\end{equation*}
is the unique generalized partial-slice monogenic extension of $f_{\underline{\eta}}$ to the whole $\Omega$.
\end{theorem}
\begin{proof}
The fact that $f(\bx_p+r \underline{\omega}):={\rm ext}(f_{\underline{\eta}})(\bx_p+ r \underline{\omega})$ is generalized partial-slice monogenic follows from the computations in the proof of Theorem \ref{Representation-Formula-SM}. Since $f(\bx_p+r \underline{\eta})=f_{\underline{\eta}}(\bx_p+r\underline{\eta})$ the identity   theorem in Theorem \ref{Identity-theorem} implies that the extension is unique.
\end{proof}
For each arbitrary but fixed $\eta\in\mathbb S$, generalized partial-slice monogenic functions can be seen as classical monogenic functions and for this reason it is possible to obtain another type of extension starting from a real analytic function defined in $\mathbb R^{p+1}$.
 For simplicity, we consider only the case of polynomials, which is sufficient for our purpose in Proposition  \ref{Ck-P} below.
\begin{definition}[CK-extension]\label{Cauchy-Kovalevska-extension}
Let $f_{0}:\mathbb{R}^{p+1} \to \mathbb{R}_{p+q}$ be a  polynomial. Define the  generalized partial-slice Cauchy-Kovalevskaya extension (CK-extension for short)  $CK[f_{0}]:  \mathbb{R}^{p+q+1}\to \mathbb{R}_{p+q}$ by
$$CK[f_{0}](\bx)= {\rm exp} (r  \underline{\omega} D_{\bx_p} ) f_{0}(\bx_p)
=\sum_{k=0}^{+\infty} \frac{r^{k}}{k!}( \underline{\omega} D_{\bx_p})^{k}f_{0}(\bx_p),$$
where $\bx=\bx_p+ r \underline{\omega}$ with $\bx_{p}\in\R^{p+1}, r \geq 0,$ and $\underline{\omega}  \in \mathbb S$.
\end{definition}

Note that $CK[f_{0}]$ is well-defined  since  the resulting  function does not depend  on $\underline{\omega}$ at $\bx_p$ and the series is finite sum for all polynomials.  In fact, it is proven in a forthcoming paper  that  the series given in Definition \ref{Cauchy-Kovalevska-extension}  still converges in $\mathbb{R}^{p+q+1}$ when the polynomial $f_0$ is replaced by any real analytic function  in $\mathbb R^{p+1}$ such that  its Taylor series with center in a point, e.g. the origin, converges   in the whole  $\mathbb{R}^{p+1}$.

\begin{theorem}
Let $f_{0}:  \mathbb{R}^{p+1} \to \mathbb{R}_{p+q}$ be a  polynomial. Then $CK[f_0]$  is the unique extension preserving generalized partial-slice monogenicity of $f_{0}$ to $\mathbb{R}^{p+q+1}$.
\end{theorem}
\begin{proof}
 First, we show that $CK[f_{0}]\in \mathcal{GSM}^{L}(\mathbb{R}^{p+q+1})$.
For any $\underline{\omega}\in \mathbb{S}$, we have
\begin{eqnarray*}
 (\underline{\omega}\partial_{r})  CK[f_{0}]_{\underline{\omega}}
 &=&\underline{\omega} \sum_{k=0}^{+\infty} \frac{r^{k}}{k!}( \underline{\omega} D_{\bx_p})^{k+1}f_{0}(\bx_p)
 \\
&=&-D_{\bx_p}\sum_{k=0}^{+\infty} \frac{r^{k}}{k!}( \underline{\omega} D_{\bx_p})^{k}f_{0}(\bx_p)
 \\
&=&-D_{\bx_p}CK[f_{0}]_{\underline{\omega}}.
\end{eqnarray*}
Hence,
$$D_{\underline{\omega}} CK[f_{0}]_{\underline{\omega}}=0.$$
By the identity   theorem in Theorem \ref{Identity-theorem},   $CK[f_{0}]$  is the unique generalized partial-slice monogenic extension of $f_{0}$.
\end{proof}

\begin{remark}
{\rm The  classical  CK-extension and     Theorem \ref{extthm} open  the way to a definition of multiplication between generalized partial-slice monogenic functions. More precisely, given $f,g\in \mathcal {GSM}^{L}(\mathbb{R}^{p+q+1}) $, one may define
$$f\ast g(\bx)= {\rm
ext}(e^{-x_{0}\mathbf{D}}(f_{\underline{\eta}}(\underline{\bx}_p+r\underline{\eta}
)g_{\underline{\eta}}(\underline{\bx}_p+r\underline{\eta} )))(\bx_p+r
\underline{\omega}) \in \mathcal {GSM}^{L}(\mathbb{R}^{p+q+1}),$$
where $\bx= \bx_p+ r \underline{\omega}$ and
$\mathbf{D}=D_{\underline{\bx}_p}+ \underline{\eta}\partial_r$.
}
\end{remark}

The validity of the Splitting Lemma (Lemma \ref{Splitting-lemma}) and of the Representation Formula  (Theorem \ref{Representation-Formula-SM}) allows to lift some main properties of monogenic functions such as Cauchy  integral formula, maximum modulus principle,  Taylor and  Laurent series expansions
to the setting of generalized partial-slice monogenic functions. To this end, we need introduce more notation.

\medskip

Consider a multi-index $\mathrm{k}=(k_0,k_1,\ldots,k_{p})\in \mathbb{N}^{p+1}$, set  $k=|\mathrm{k}|=k_0+k_1+\cdots+k_{p}$, and $\mathrm{k}!=k_0 !k_1 !\cdots k_{p}!$. Let  $\bx'=(x_0,x_1,\ldots,x_p,r)=(\bx_p,r) \in \mathbb{R}^{p+2}$, and set $\bx'_{\diamond}=(\bx_p,-r)$, $\bx_p^{\mathrm{k}}= x_0^{k_0} x_1^{k_1}\ldots x_p^{k_p}$.  Let $\underline{\eta} \in \mathbb{S}$ be arbitrary but fixed. The so-called (left) Fueter variables are defined as
 $$ z_{\ell}=  x_{\ell}+r \underline{\eta} e_\ell, \ \ell=0,1,\ldots,p.$$
Similarly, the so-called right  Fueter variables are defined as
 $$ z_{\ell}^{R}=  x_{\ell}+r e_\ell \underline{\eta}, \ \ell=0,1,\ldots,p.$$
An easy calculation shows that
$$(\sum_{i=0}^{p}e_i\partial_{x_i}+\underline{\eta}\partial_{r})z_{\ell}= z_{\ell}^{R}(\sum_{i=0}^{p}e_i\partial_{x_i}+\underline{\eta}\partial_{r})=0.
$$
Meanwhile,
$$z_{\ell}(\sum_{i=0}^{p}e_i\partial_{x_i}+\underline{\eta}\partial_{r})=  (\sum_{i=0}^{p}e_i\partial_{x_i}+\underline{\eta}\partial_{r})z_{\ell}^{R}=2e_\ell.$$
Note that these variables are related to the classical Fueter variables in $\mathbb{R}^{p+2}$, see e.g. \cite[Section 6.1]{Gurlebeck}, if one writes $\underline{\eta}e_\ell$ on the right, namely
$$ z_{\ell}= ( r-x_{\ell} \underline{\eta} e_\ell )\underline{\eta}e_\ell, \ \ell=0,1,\ldots,p.$$

\begin{definition}
Let us consider $(j_1,j_2,\ldots, j_k)$ to be an alignment of integers with $0\leq j_\ell\leq p$ for any $\ell=0,\ldots,k$ and assume that the number of $0$ in the alignment is $k_0$, the number of $1$ is $k_1$ and the number of $p$ is $k_{p}$, where $k_0+k_1+\ldots+k_{p}=k$.
We define
$$\mathcal{P}_{ \underline{\eta},\mathrm{k}}  ( \bx')=\frac{1}{k!}\sum_{(\sigma(j_1),\sigma(j_2),\ldots, \sigma(j_k))\in \mathcal{P}(j_1,j_2,\ldots, j_k)} z_{\sigma(j_1)} z_{\sigma(j_2)}\cdots z_{\sigma(j_k)},$$
where the sum is computed over the $\dfrac{k!}{\mathrm{k}!}$ different permutations $\sigma$ of $k_\ell$ elements equal to $\ell=0,1,\ldots, p$. When $\mathrm{k}=(0,\ldots,0)=0$ we set $\mathcal{P}_{ \underline{\eta},0}(\bx')=1$.

Similarly, we can define $\mathcal{P}^{R}_{ \underline{\eta},\mathrm{k}}  ( \bx')$ when $z_{\ell}$ are replaced by $z_{\ell}^{R}$.
\end{definition}

%

\begin{proposition}
For any $\eta\in\mathbb S$ and any multi-index $\mathrm{k}\in\mathbb{N}^{p+1}$, the polynomials $\mathcal{P}_{ \underline{\eta},\mathrm{k}}(\bx')$ and $\mathcal{P}^{R}_{ \underline{\eta},\mathrm{k}}  (\bx')$ satisfy, respectively,
\begin{equation}\label{Pkappa}
(\sum_{i=0}^{p}e_i\partial_{x_i}+\underline{\eta}\partial_{r})\mathcal{P}_{ \underline{\eta},\mathrm{k}}  (\bx')=0, \quad \mathcal{P}^{R}_{ \underline{\eta},\mathrm{k}}  (\bx')(\sum_{i=0}^{p}e_i\partial_{x_i}+\underline{\eta}\partial_{r})=0.
\end{equation}
\end{proposition}

\begin{proof} Following the reasoning of the proof of \cite[Theorem 6.2]{Gurlebeck} with suitable modifications, we only prove the first result by dividing the proof in steps.   The second  result can be  achieved  in the same way.

Step 1. Let us set $\mathrm{\epsilon}_i\in\mathbb{N}^{p+1}$ be the multi-index with $1$ in the $i$-th position and $0$ elsewhere. Then we have
$$
k \mathcal{P}_{ \underline{\eta},\mathrm{k}}(\bx')  =\sum_{i=0}^p  z_i \mathcal{P}_{ \underline{\eta},\mathrm{k}-\mathrm{\epsilon}_i}(\bx')=
\sum_{i=0}^p  \mathcal{P}_{ \underline{\eta},\mathrm{k}-\mathrm{\epsilon}_i}(\bx')z_i.
$$
(Note that when  there is a negative integer in the multi-index $\mathrm{k}-\mathrm{\epsilon}_i$, we set $\mathcal{P}_{ \underline{\eta},\mathrm{k}-\mathrm{\epsilon}_i}(\bx')=0$).
To see this, we note that each summand in $\mathcal{P}_{ \underline{\eta},\mathrm{k}}$ contains $z_i$ is the first position so we can write
$$
k \mathcal{P}_{ \underline{\eta},\mathrm{k}}(\bx') =\sum_{i=0}^p  z_i \mathcal{R}_{ \underline{\eta},\mathrm{k}, i}(\bx')
$$
for some polynomials $ \mathcal{R}_{ \underline{\eta},\mathrm{k}, i}$ which are homogeneous of degree $(k-1)$. However, each of the $ \mathcal{R}_{ \underline{\eta},\mathrm{k}, i}$ must contain all the permutations of the remaining $(k-1)$ elements $z_j$, except $z_i$ which has been put in the first position. Thus, since we multiplied the left hand side by $k$, the factor in front is $1/(k-1)!$ and this means that  $ \mathcal{R}_{ \underline{\eta},\mathrm{k}, i}=  \mathcal{P}_{ \underline{\eta},\mathrm{k-\mathrm{\epsilon}_i}}$  and  the assertion on the first equality follows. The second equality follows in a similar way.

Step 2. As a consequence of Step 1, writing $z_i=x_i + \underline{\eta} e_i r$ and taking into account that $x_i$ commutes with $\mathcal{P}_{ \underline{\eta},\mathrm{k}-\mathrm{\epsilon}_i}(\bx')$ and so it can be cancelled at both hand sides, we deduce that
 $$
 \sum_{i=0}^p \underline{\eta} e_i\mathcal{P}_{ \underline{\eta},\mathrm{k}-\mathrm{\epsilon}_i}(\bx')
= \sum_{i=0}^p\mathcal{P}_{ \underline{\eta},\mathrm{k}-\mathrm{\epsilon}_i}(\bx')\underline{\eta} e_i.$$

Step 3. We compute the derivatives of $\mathcal{P}_{ \underline{\eta},\mathrm{k}}$ with respect to $x_j$, $j=0,\ldots, p$, by induction on $k$, and we prove that
$$
\partial_{x_j}\mathcal{P}_{ \underline{\eta},\mathrm{k}}(\bx')=  \mathcal{P}_{ \underline{\eta},\mathrm{k}-\mathrm{\epsilon}_j}(\bx').
$$
The assertion is true for $\mathrm{k}=(0,\ldots, 0)$ since in that case $\mathcal{P}_{ \underline{\eta},\mathrm{k}-\mathrm{\epsilon}_j}(\bx')=0$ and $\partial_{x_j} \mathcal{P}_{ \underline{\eta},\mathrm{k}}(\bx')=0$. Assume that the assertion holds for $(k-1)$ and we prove it for the degree of homogeneity $k$ by using Step 1:
\[
\begin{split}
k\partial_{x_j}\mathcal{P}_{ \underline{\eta},\mathrm{k}}(\bx')&= \sum_{i=0}^p   \partial_{x_j}\left(z_i\mathcal{P}_{ \underline{\eta},\mathrm{k}-\mathrm{\epsilon}_i}(\bx')\right)\\
&= \sum_{i=0}^p  (\partial_{x_j}z_i) \mathcal{P}_{ \underline{\eta},\mathrm{k}-\mathrm{\epsilon}_i}(\bx')
+ \sum_{i=0}^p   z_i \mathcal{P}_{ \underline{\eta},\mathrm{k}-\mathrm{\epsilon}_i-\mathrm{\epsilon}_j}(\bx')\\
&=  \mathcal{P}_{ \underline{\eta},\mathrm{k}-\mathrm{\epsilon}_j}(\bx')
+(k-1)   \mathcal{P}_{ \underline{\eta},\mathrm{k}-\mathrm{\epsilon}_j}(\bx')\\
&= k   \mathcal{P}_{ \underline{\eta},\mathrm{k}-\mathrm{\epsilon}_j}(\bx'),
\end{split}
\]
and the formula follows.

Step 4. We compute $\partial_r \mathcal{P}_{ \underline{\eta},\mathrm{k}}(\bx')$ as
\begin{equation}\label{DPk}
\partial_r \mathcal{P}_{ \underline{\eta},\mathrm{k}}(\bx')= \sum_{j=0}^p  \underline{\eta} e_j \mathcal{P}_{ \underline{\eta},\mathrm{k}-\mathrm{\epsilon}_j}(\bx')= \sum_{j=0}^p \underline{\eta} e_j\partial_{x_j} \mathcal{P}_{ \underline{\eta},\mathrm{k}}(\bx').
\end{equation}
We proceed by induction on the degree $k$. The assertion holds for $k=0$, so we assume it for $(k-1)$ and   prove it for $k$. In fact, it holds that
\[
\begin{split}
k\partial_r \mathcal{P}_{ \underline{\eta},\mathrm{k}}(\bx')&= \sum_{i=0}^p   \partial_{r}\left(\mathcal{P}_{ \underline{\eta},\mathrm{k}-\mathrm{\epsilon}_i}(\bx') z_i \right)\\
&=\sum_{i,j=0}^p   \underline{\eta}e_j\mathcal{P}_{ \underline{\eta},\mathrm{k}-\mathrm{\epsilon}_i-\mathrm{\epsilon}_j}(\bx') z_i
+ \sum_{i=0}^p  \mathcal{P}_{ \underline{\eta},\mathrm{k}-\mathrm{\epsilon}_i}(\bx') \underline{\eta}e_i\\
&=(k-1)\sum_{j=0}^p  \underline{\eta}e_j \mathcal{P}_{ \underline{\eta},\mathrm{k}-\mathrm{\epsilon}_j}(\bx')+ \sum_{i=0}^p  \underline{\eta}e_i\mathcal{P}_{ \underline{\eta},\mathrm{k}-\mathrm{\epsilon}_i}(\bx')\\
&=k \sum_{j=0}^p   \underline{\eta}e_j\mathcal{P}_{ \underline{\eta},\mathrm{k}-\mathrm{\epsilon}_j}(\bx'),
\end{split}
\]
where we used also Step 1 and Step 2 to get the third equality.

Step 5. By multiplying on left by $\underline{\eta}$ both sides of \eqref{DPk}, we get \eqref{Pkappa} and this concludes the proof.
\end{proof}

By letting $\underline{\eta}$ to vary in the sphere $\mathbb S$ we obtain the definition of generalized partial-slice monogenic polynomial
$\mathcal{P}_{ \mathrm{k}}  (\bx),$
where $\bx =\bx_p+r \underline{\omega}$ and the $\underline{\omega}$ is omitted as subscript in $\mathcal{P}_{ \mathrm{k}}$. The fact that $\mathcal{P}_{ \mathrm{k}} $ (resp. $\mathcal{P}^{R}_{\mathrm{k}}$) is left (resp. right) generalized partial-slice monogenic follows from \eqref{Pkappa}. Hence, we call $\mathcal{P}_{\mathrm{k}}  $ (resp. $\mathcal{P}^{R}_{\mathrm{k}}$) Fueter polynomials for  left  (resp. right) generalized partial-slice monogenic functions. We note that from the Representation Formula, that in this case can be directly verified, one also has:
\begin{eqnarray*} \label{ReprPk}
 \mathcal{P}_{ \mathrm{k}}  (\bx)
 &=&\frac{1}{2}(1-\underline{\omega}\underline{\eta}) \mathcal{P}_{\underline{\eta},\mathrm{k}}(\bx')+
\frac{1}{2}(1+ \underline{\omega}\underline{\eta}) \mathcal{P}_{\underline{\eta},\mathrm{k}}(\bx'_{\diamond}).
\end{eqnarray*}

\begin{proposition} \label{Ck-P}
  $CK[\bx_{p}^{\mathrm{k}}](\bx)= \mathrm{k}!\mathcal{P}_{ \mathrm{k}}(\bx).$     In particular, $z_{\ell}=  CK[x_{\ell}](\bx_{p}+\underline{\eta} r), \ \ell=0,1,\ldots,p.$
\end{proposition}
\begin{proof}
Observing that $\mathrm{k}!\mathcal{P}_{ \underline{\eta},\mathrm{k}}  ( \bx_{p})=\bx_{p} ^{\mathrm{k}}$,  the identity  theorem in Theorem \ref{Identity-theorem} gives the statement.
\end{proof}


 Let $\mathrm{k}=(k_0,k_1,\ldots,k_{p})\in\mathbb{N}^{p+1}$ and  $f:\Omega\subseteq \mathbb{R}^{p+q+1}\rightarrow \mathbb{R}_{p+q}$ be a function with continuous $k$-times partial derivatives, where $k=|\mathrm{k}|$,  define the  \textit{partial derivative} $\partial_{\mathrm{k}}$ as
\begin{equation*}
\partial_{\mathrm{k}}f(\bx)=\frac{\partial^{k} }{\partial_{x_{0}}^{k_0}\partial_{x_{1}}^{k_1}\cdots\partial_{x_{p}}^{k_p} }f(\bx_p+\underline{\bx}_q), \quad \bx=\bx_p+\underline{\bx}_q.
\end{equation*}

For every $\underline{\omega}\in \mathbb{S}$, let us  define the  \textit{slice operator} $\partial_{\underline{\omega},\mathrm{k}}$ as
\begin{equation*}\label{slice-operator}
\partial_{\underline{\omega},\mathrm{k}}f(\bx)=\frac{\partial^{k} }{\partial_{x_{0}}^{k_0}\partial_{x_{1}}^{k_1}\cdots\partial_{x_{p}}^{k_p} }f(\bx_p+\underline{\omega}r).
\end{equation*}

The terminology slice operator, which is a slight abuse, comes form the fact that if we consider the case $(p,q)=(0,n)$, then $\partial_{\mathrm{k}}=\partial_{x_{0}}^{k_0}$ equals the standard slice derivative for slice monogenic functions.

\begin{proposition}\label{preserving-slice-monogenic}
Let $\Omega\subseteq \mathbb{R}^{p+q+1}$ be a  domain and $f:\Omega\rightarrow \mathbb{R}_{p+q}$ be a  generalized partial-slice monogenic function.  Then, for any $\by\in \mathbb{R}^{p+1}$ and  $\mathrm{k}\in\mathbb{N}^{p+1}$, we have
$$f(\cdot-\by) \in \mathcal {GSM}(\Omega+\{\by\}), \quad \partial_{\mathrm{k}}f \in \mathcal {GSM}(\Omega),$$
where $\Omega +\{\by\}=\{\bx\in \mathbb{R}^{p+q+1} : \bx-\by\in \Omega\}$.
\end{proposition}

\begin{lemma}\label{lemmakth}
Let $\underline{\eta}\in\mathbb S$ and consider the polynomial $P$ defined  in $\mathbb{R}^{p+q+1}$  satisfying that its restriction $P_{\underline{\eta}}$ is homogeneous  of degree $k$ and
$$(\sum_{i=0}^{p}e_i\partial_{x_i}+\underline{\eta}\partial_{r})P(\bx_p+\underline{\eta} r)=0.$$
Then we have
 $$P(\bx_p+\underline{\eta} r)= \sum_{|\mathrm{k}|=k} \mathcal{P}_{ \underline{\eta},\mathrm{k}}  ( \bx') a_{\mathrm{k}},\quad  a_{\mathrm{k}}= \partial_{\mathrm{k}} P(0). $$
\end{lemma}

\begin{proof}
We fix $\underline{\eta}\in\mathbb S$ and consider a generic polynomial, homogeneous of degree $k$,
 satisfying
 $$
 (\sum_{i=0}^{p}e_i\partial_{x_i}+\underline{\eta}\partial_{r})P(\bx_p+\underline{\eta} r)=0,
 $$
 from which we deduce
 \begin{equation}\label{Poly}
 \partial_{r}P(\bx_p+\underline{\eta} r)=\sum_{i=0}^{p} \underline{\eta} e_i\partial_{x_i} P(\bx_p+\underline{\eta} r) .
 \end{equation}
 Since $P_{\underline{\eta}}$ is homogeneous of degree $k$, $P_{\underline{\eta}}$ satisfies also
 $$
 \sum_{i=0}^p x_i\partial_{x_i} P_{\underline{\eta}} + r \partial_r P_{\underline{\eta}}= k P_{\underline{\eta}},
 $$
 and substituting \eqref{Poly} in this last expression we get, recalling that $e_0=1$, that
 $$
 k P(\bx_p+\underline{\eta} r)=  \sum_{i=0}^p (x_i+\underline{\eta} e_i r) \partial_{x_i}P(\bx_p+\underline{\eta} r)= \sum_{i=0}^p z_i  \partial_{x_i}P(\bx_p+\underline{\eta} r).
 $$
  We now iterate the procedure for the derivatives $\partial_{x_i}P_{\underline{\eta}}$, $i=0,\,\ldots, p,$ which are homogeneous polynomials of degree $(k-1)$ and also also   in the kernel of the operator $(D_{\bx_p}+\underline{\eta}\partial r)$. After $k$ iterations we obtain:
  $$
  k! P(\bx_p+\underline{\eta} r)=  \sum_{i_1,\ldots, i_k=0}^p z_{i_1}\ldots z_{i_k}\frac{\partial^k}{\partial_{x_{i_1}}\ldots \partial_{x_{i_k}}}P(\bx_p+\underline{\eta} r).
  $$
  Since the order of derivation is not important, we can group all the derivatives of the form
  $\frac{\partial^{k} P }{\partial_{x_{0}}^{k_0}\partial_{x_{1}}^{k_1}\cdots\partial_{x_{p}}^{k_p} }$
  and we obtain
  \[
  \begin{split}
  P(\bx_p+\underline{\eta} r)&=  \frac{1}{k! }  \sum_{i_1,\ldots, i_k=0}^p z_{i_1}\ldots z_{i_k}\frac{\partial^k}{\partial_{x_{i_1}}\ldots \partial_{x_{i_k}}}P(\bx_p+\underline{\eta} r)\\
  &=\sum_{|\mathrm{k}|=k, \ \mathrm{k}=(k_0,k_1,\ldots,k_{p})} \mathcal{P}_{\mathrm{k}}(\bx_p+\underline{\eta} r) \frac{\partial^{k}   }{\partial_{x_{0}}^{k_0}\partial_{x_{1}}^{k_1}\cdots\partial_{x_{p}}^{k_p} }P(\bx_p+\underline{\eta} r)\\
  &=\sum_{|\mathrm{k}|=k} \mathcal{P}_{\mathrm{k}}(\bx_p+\underline{\eta} r) \partial_{\mathrm{k}} P(0),
  \end{split}
  \]
  where last equality follows form the fact that $P_{\underline{\eta}}$ has degree $k$ and we are differentiating it $k$ times.
  The assertion now follows.
\end{proof}

Let us denote by  $B(\by,\rho)=\{\bx \in \mathbb{R}^{p+q+1}: |\by-\bx|<\rho \}$  the ball centered in $\by\in \mathbb{R}^{p+q+1}$ with radius $\rho>0$.
\begin{lemma} \label{Taylor-lemma-left}
Let $f: B=B(0,\rho)\rightarrow \mathbb{R}_{p+q}$ be a generalized partial-slice monogenic function. For  any $\bx\in B$,  there exists   $\underline{\omega}\in \mathbb{S}$ such that $\bx=\bx_p+r \underline{\omega} \in B_{\underline{{\omega}} }$, and
$$f(\bx)=  \sum_{k=0}^{+\infty} \sum_{|\mathrm{k}|=k} \mathcal{P}_{ \underline{\omega},\mathrm{k}}  (\bx')\partial_{ \mathrm{k}}  f(0), \quad \bx'=(\bx_p,r ), $$
where the series converges uniformly on compact subsets of $B_{\underline{{\omega}} }$.
\end{lemma}
\begin{proof}
Let us fix $\underline{\omega}=I_{p+1}$ and let us complete to a basis of $\mathbb{R}_{p+q}$ as in Lemma \ref{Splitting-lemma}. Then the restriction $f_{\underline{\omega}}$ can be written as
\begin{equation}\label{equalem} f_{\underline{\omega}}(\bx_p+r\underline{\omega})=\sum_{A}F_{A}(\bx_p+r\underline{\omega})I_{A}, \quad \bx_p+r\underline{\omega}\in B_{\underline{\omega}},
\end{equation}
where $I_{A}=I_{i_{1}}\cdots I_{i_{s}},  A=\{i_{1},\ldots, i_{s}\} \subseteq \{p+2,\ldots,p+q\}$.
All components $F_A$, being monogenic, are also real analytic. Hence write
the components $F_A$ in the form
$$F_A(\bx_p+r\underline{\omega})=\sum_{k=0}^{+\infty} F_{A,k}(\bx_p+r\underline{\omega}),$$
where the series converges uniformly on compact subsets in $B_{\underline{\omega}}$ and  $F_{A,k}$ are homogeneous polynomials of degree $k$  such that $(\sum_{i=0}^{p}e_i\partial_{x_i}+\underline{\omega}\partial_{r})F_{A,k}=0$.  Lemma \ref{lemmakth} implies that
$$F_{A,k}(\bx_p+r\underline{\omega})=
\sum_{|\mathrm{k}|=k}   \mathcal{P}_{\underline{\omega},\mathrm{k}}(\bx_p+r\underline{\omega}) \partial_{\mathrm{k}} F_{A,k}(0).$$
Substituting this expression into \eqref{equalem}, we obtain
\begin{eqnarray*}
f_{\underline{\omega}}(\bx_p+r\underline{\omega})
 &=&\sum_{A=\{i_{1},\ldots, i_{s}\} \subseteq \{p+2,\ldots,p+q\}}F_{A}(\bx_p+r\underline{\omega})I_{A}
 \\
&=&\sum_{A}\sum_{k=0}^{+\infty} \sum_{|\mathrm{k}|=k}\mathcal{P}_{\underline{\omega},\mathrm{k}}(\bx_p+r\underline{\omega}) \partial_{ \mathrm{k}} F_A(0) I_{A}
\\
&=&\sum_{k=0}^{+\infty} \sum_{|\mathrm{k}|=k}\mathcal{P}_{\underline{\omega},\mathrm{k}}(\bx_p+r\underline{\omega}) \partial_{\mathrm{k}}f(0).
\end{eqnarray*}
The statement follows.
\end{proof}

 \begin{remark}\label{Taylor-lemma-right}
 Using the same method as in Lemma \ref{Taylor-lemma-left}, one can prove that, for  right generalized partial-slice monogenic functions $f: B=B(0,\rho)\rightarrow \mathbb{R}_{p+q}$,
$$f(\bx)=  \sum_{k=0}^{+\infty} \sum_{|\mathrm{k}|=k}  \partial_{ \mathrm{k}}  f(0)  \mathcal{P}^{R}_{ \underline{\omega},\mathrm{k}}  (\bx'), \quad \bx \in B_{\underline{{\omega}} }.$$
\end{remark}

\begin{theorem}[Taylor series]\label{Taylor-theorem}
Let $f: B=B(0,\rho)\rightarrow \mathbb{R}_{p+q}$ be a generalized partial-slice monogenic function. For  any $\bx\in B$, we have
$$f(\bx)=  \sum_{k=0}^{+\infty} \sum_{|\mathrm{k}|=k} \mathcal{P}_{\mathrm{k}}(\bx)  \partial_{\mathrm{k}} f (0),  $$
where the series converges uniformly on compact subsets of $B$.
\end{theorem}
\begin{proof}
Let $\bx\in B$ and   $\underline{\eta} \in \mathbb{S}$, Write  $\bx=\bx_p+r  \underline{\omega}$ for some  $\bx_{p}\in\R^{p+1}, r\geq0,$ and $ \underline{\omega}\in \mathbb{S}$ and set $\bx'=(\bx_p,r ),\bx'_{\diamond}=(\bx_p,-r)$. Since the ball $B$ is a p-symmetric slice domain, we can use  Theorem \ref{Representation-Formula-SM} and combining with the previous Lemma \ref{Taylor-lemma-left}, it follows that
 \begin{eqnarray*}
 f(\bx)
&=& \frac{ 1}{2} (1- \underline{\omega}\underline{\eta})    f(\bx_p+r \underline{\eta})+
 \frac{ 1}{2} (1+ \underline{\omega}\underline{\eta} )  f(\bx_p-r \underline{\eta} )
 \\
 &=& \frac{ 1}{2} (1- \underline{\omega} \underline{\eta}) \sum_{k=0}^{+\infty} \sum_{|\mathrm{k}|=k} \mathcal{P}_{\underline{\eta}, \mathrm{k}}(\bx')  \partial_{\mathrm{k}} f (0)
 \\
  & &+\frac{ 1}{2} (1+ \underline{\omega}\underline{\eta} ) \sum_{k=0}^{+\infty} \sum_{|\mathrm{k}|=k} \mathcal{P}_{\underline{\eta},\mathrm{k}}(\bx'_{\diamond})  \partial_{\mathrm{k}} f (0)
 \\
 &=& \frac{ 1}{2} \sum_{k=0}^{+\infty} \sum_{|\mathrm{k}|=k} \big((1- \underline{\omega}\underline{\eta}) \mathcal{P}_{\underline{\eta},\mathrm{k}}(\bx')+(1+ \underline{\omega}\underline{\eta}) \mathcal{P}_{\underline{\eta},\mathrm{k}}(\bx'_{\diamond})\big)  \partial_{\mathrm{k}} f(0)
  \\
   &=&\sum_{k=0}^{+\infty} \sum_{|\mathrm{k}|=k} \mathcal{P}_{\mathrm{k}}(\bx_p+r  \underline{\omega} )  \partial_{\mathrm{k}} f(0) ,
   \end{eqnarray*}
 where, by the arbitrariness of $\eta\in\mathbb S$, the series converges uniformly on compact subsets of $B$.
 \end{proof}

\begin{remark}{\rm
In Theorem \ref{Taylor-theorem},  the Taylor series expansion for generalized partial-slice monogenic functions is centered at $\by=0$, but it could be centered at $\by \in B\cap \mathbb{R}^{p+1}$ by Proposition \ref{preserving-slice-monogenic}.  However,  it is an open problem to establish if there is a formula centered at  $\by \in  \mathbb{R}^{p+q+1}$ which is valid in open sets (in the Euclidean sense). Already in the case of quaternionic slice regular functions, the Taylor series centered at an arbitrary point is based on a different type of expansion, see  \cite[Theorem 4.1]{Stoppato}.}
\end{remark}
We now prove the forms of the Cauchy formulas, the first one with limited validity since it holds on ``slices" and the second one which is valid on open sets and is based on a kernel which is generalized partial-slice monogenic. To this end, we recall the function in Example \ref{Cauchy-kernel-example}:
$$E(\bx) =\frac{1}{\sigma_{p+1}}\frac{\overline{\bx}}{|\bx|^{p+2}} \in \mathcal {GSM}(\mathbb{R}^{p+q+1}\setminus \{0\}).$$

\begin{theorem}[Cauchy  formula, I]\label{Cauchy-formula-monogenic}
Let $\Omega\subseteq \mathbb{R}^{p+q+1}$ be a  p-symmetric slice domain and $f:\Omega\rightarrow \mathbb{R}_{p+q}$ be a  generalized partial-slice monogenic function. If  $U$ is a domain in  $\mathbb{R}^{p+q+1}$ such that  $U_{\underline{{\omega}} }\subset \Omega_{\underline{{\omega}}} $ is a   bounded domain in $\mathbb{R}^{p+2}$ with    smooth boundary $\partial U_{\underline{{\omega}}}\subset\Omega_{\underline{\omega}}$ for some $\underline{{\omega}}\in \mathbb{S}$,  then
$$f(\bx)=\int_{\partial U_{\underline{{\omega}}}} E(\by- \bx) n(\by)f(\by) dS(\by), \quad  \bx \in U_{\underline{{\omega}}},   $$
where  $n(\by)=\sum_{i=0}^{p}n_i(\by) e_i+n_{p+1}(\by)\underline{{\omega}}$ is the unit exterior normal to $\partial U_{\underline{{\omega}}}$ at $\by$ and $dS$  stands  for  the   classical   Lebesgue surface  element in $\mathbb{R}^{p+2}$.
\end{theorem}
\begin{proof}
Under the assumptions, the Splitting Lemma \ref{Splitting-lemma} guarantees that there exist $2^{q-1}$ monogenic functions $F_{A}:\Omega_{\underline{\omega}}\subseteq\mathbb{R}^{p+2} \rightarrow \mathbb{R}_{p+1}$  such that
\begin{equation*}
f(\bx)=\sum_{A=\{i_{1},\ldots, i_{s}\} \subseteq \{p+2,\ldots,p+q\}}F_{A}(\bx)I_{A}, \quad \bx=\bx_p+r\underline{\omega}\in \Omega_{\underline{\omega}}.
\end{equation*}
The Cauchy  integral formula (see e.g. \cite[Corollary 9.6]{Brackx}) for monogenic functions applied to each $F_{A}$  gives that
$$F_{A}(\bx)=\int_{\partial U_{\underline{{\omega}}}} E(\by- \bx) n(\by)F_{A}(\by) dS(\by), \quad  \bx \in U_{\underline{{\omega}}}\subset\Omega_{\underline{\omega}}, $$
where  $n(\by)$ and $dS$ are as in the statement.

Hence, \begin{eqnarray*}
 f(\bx)
&=&\sum_{A=\{i_{1},\ldots, i_{s}\} \subseteq \{p+2,\ldots,p+q\}}\Big(\int_{\partial U_{\underline{{\omega}}}} E(\by- \bx) n(\by)F_{A}(\by) dS(\by) \Big)I_{A}
 \\
 &=&\int_{\partial U_{\underline{{\omega}}}} E(\by- \bx) n(\by) \sum_{A} F_{A}(\by) I_{A} dS(\by)
 \\
 &=&\int_{\partial U_{\underline{{\omega}}}} E(\by- \bx) n(\by)  f(\by)  dS(\by),
\end{eqnarray*}
and the proof is complete.
\end{proof}

To prove our next results, we need more notation.

Let $\underline{\eta} \in \mathbb{S}$ and define the  functions $\mathcal{Q}_{\underline{\eta},\mathrm{k}}: \mathbb{R}^{p+2}\setminus \{0\}\rightarrow \mathbb{R}_{p+q}$, homogeneous of degree $-(k+p+1)$, as
$$\mathcal{Q}_{\underline{\eta},\mathrm{k}}(\bx')=(-1)^{k} \partial_{\mathrm{k}} E(\bx_p+r  \underline{\eta}).$$
From Example  \ref{Cauchy-kernel-example} and Proposition  \ref{preserving-slice-monogenic}, we can define the left (and right)  generalized partial-slice monogenic functions $\mathcal{Q}_{ \mathrm{k}}: \mathbb{R}^{p+q+1}\setminus \{0\}\rightarrow \mathbb{R}_{p+q}$  by letting $\eta$ to vary in $\mathbb S$, or equivalently, by setting
$$\mathcal{Q}_{ \mathrm{k}}  (\bx)=(-1)^{k}  \partial_{\mathrm{k}} E(\bx)=\frac{1}{2}(1- \underline{\omega}\underline{\eta}) \mathcal{Q}_{\underline{\eta},\mathrm{k}}(\bx')+\frac{1}{2}(1+ \underline{\omega}\underline{\eta}) \mathcal{Q}_{\underline{\eta},\mathrm{k}}(\bx'_{\diamond}), \quad \bx=\bx_p+r  \underline{\omega}.$$

It is important to note that, by the noncommutativity, for a generic $\by \in \mathbb{R}^{p+q+1}\setminus  \mathbb{R}^{p+1}$, we have that
$$E_{\by}(\cdot) :=E(\cdot, \by)=E(\by- \cdot)  \notin \mathcal {GSM}(\mathbb{R}^{p+q+1}\setminus \{\by\}).$$
However, for any fixed $\by \in \mathrm{H}_{\underline{\eta}}$ with $\underline{\eta}\in \mathbb{S}$,  the restriction of the function $E(\cdot, \by)$ to $\mathrm{H}_{\underline{\eta}}  \setminus \{\by\}$ is both left and right in the kernel of the operator $(D_{\bx_p}+{\underline{\eta}}\partial r)$.  Likewise, for any fixed $\bx \in \mathrm{H}_{\underline{\eta}}$ with $\underline{\eta}\in \mathbb{S}$, the restriction of the function $E(\bx, \cdot)$ to $\mathrm{H}_{\underline{\eta}}  \setminus \{\bx\}$ is both left and right in the kernel of the operator $(D_{\bx_p}+{\underline{\eta}}\partial r)$. Thus, using Theorem \ref{extthm},  we introduce a function that uniquely extends $E_{\by}(\cdot) $ and is generalized partial-slice monogenic as follows.
\begin{definition} Given $\by \in \mathbb{R}^{p+q+1}$,  we call the function $\mathcal{E}_{\by}(\cdot)$ \textit{left generalized partial-slice   Cauchy kernel} defined by
\begin{equation*}\label{slice-Cauchy-kernel}
\mathcal{E}_{\by}(\bx)=\frac{1}{2}(1-\underline{\omega}\underline{\eta})E_{\by}(\pi_{\by}(\bx))+
\frac{1}{2}(1+\underline{\omega}\underline{\eta})E_{\by}( \pi_{\by}(\bx)_{\diamond}),
\end{equation*}
where $\bx=\bx_p+r  \underline{\omega},$ $\by=\by_p+\widetilde{r}  \underline{\eta},
\pi_{\by}(\bx)=\bx_p+r \underline{\eta} $  and $\pi_{\by}(\bx)_{\diamond}=\bx_p-r\underline{\eta}$.

Likewise, we can define the \textit{right generalized partial-slice   Cauchy kernel}
 \begin{equation*} \mathcal{E}_{\by}^{R}(\bx)=\frac{1}{2}E_{\by}(\pi_{\by}(\bx))(1-\underline{\eta}\underline{\omega})+
\frac{1}{2}E_{\by}( \pi_{\by}(\bx)_{\diamond})(1+\underline{\eta}\underline{\omega}).
\end{equation*}
\end{definition}

\begin{remark}\label{Cauchy-Ker-expansion}
{\rm Recall that $[\by]=\{\by_p+r\underline{\omega}, \ \underline{\omega}\in\mathbb S\}$ and observe that $\mathcal{E}_{\by}(\cdot)$  is left generalized partial-slice monogenic in $\mathbb{R}^{p+q+1}\setminus [\by]$. If we fix $\bx$ and treat the parameter $\by$ as a variable, then $\mathcal{E}_{\cdot}(\bx)$ is right  generalized partial-slice monogenic in $\mathbb{R}^{p+q+1}\setminus [\bx]$. For $|\bx|<|\by|$, it admits the series expansion by Theorem \ref{Taylor-theorem}
\begin{equation*}
 \mathcal{E}_{\by}(\bx)=  \sum_{k=0}^{+\infty} \sum_{|\mathrm{k}|=k} \mathcal{P}_{\mathrm{k}}(\bx)\mathcal{Q}_{\mathrm{k}}(\by).
\end{equation*}
Furthermore, by Lemma \ref{Taylor-lemma-left} and Remark  \ref{Taylor-lemma-right}, for $|\bx|<|\by|$ with $\bx, \by  \in \mathrm{H}_{\underline{\eta}}$ for some $\underline{\eta}\in \mathbb{S}$,   we have
\begin{equation*}
 \mathcal{E}_{\by}(\bx)=E_{\by}(\bx)=\sum_{k=0}^{+\infty} \sum_{|\mathrm{k}|=k} \mathcal{P}_{\mathrm{k}}(\bx)\mathcal{Q}_{\mathrm{k}}(\by)=\sum_{k=0}^{+\infty} \sum_{|\mathrm{k}|=k}\mathcal{Q}_{\mathrm{k}}(\by) \mathcal{P}^{R}_{\mathrm{k}}(\bx).
\end{equation*}}
\end{remark}
A general Cauchy  formula now follows:
\begin{theorem}[Cauchy  formula, II]\label{Cauchy-formula-slice}
Let $\Omega\subseteq \mathbb{R}^{p+q+1}$ be a  p-symmetric slice domain and $f:\Omega\rightarrow \mathbb{R}_{p+q}$ be a  generalized partial-slice monogenic function. Given any $\underline{\eta} \in \mathbb{S}$,  let $U_{\underline{\eta} }\subset \Omega_{\underline{\eta}} $ be a   bounded domain in $\mathbb{R}^{p+2}$ with smooth boundary $\partial U_{\underline{\eta}}\subset\Omega_{\underline{\eta}}$. Then
$$f(\bx)=\int_{\partial U_{\underline{\eta}}} \mathcal{E}_{\by}(\bx) n(\by)f(\by) dS(\by), \quad  \bx \in U,   $$
where  the integral does not depend on the choice of $\underline{\eta}$, $n(\by)=\sum_{i=0}^{p}n_i(\by) e_i+n_{p+1}(\by)\underline{\eta}$ is the unit exterior normal to $\partial U_{\underline{\eta}}$ at $\by$ and $dS$  stands  for  the   classical   Lebesgue surface  element in $\mathbb{R}^{p+2}$.
\end{theorem}
\begin{proof}
From Theorem \ref{Cauchy-formula-monogenic} we can write a Cauchy formula which is valid on $U_{\underline{\eta}}$.
The result now follows by applying the extension theorem and the independence of $\underline{\eta}$ is assured by the identity theorem.
\end{proof}

In view of Remark \ref{Cauchy-Ker-expansion}, the Cauchy  formula II can be used to establish a Laurent expansion:
\begin{theorem}[Laurent series]\label{Laurent}
For $0\leq \rho_{1}< \rho<\rho_{2}$ and $\underline{\eta}\in \mathbb{S}$,   let  $\mathsf{A}=\{\bx \in \mathbb{R}^{p+q+1}: \rho_{1}< | \bx|<\rho_{2}\}$ and $S=\{\bx \in \mathbb{R}^{p+q+1}:  | \bx|=\rho\}$. If $f:\ \mathsf{A}\rightarrow \mathbb{R}_{p+q}$ is a  generalized partial-slice monogenic function, then
$$f(\bx)=  \sum_{k=0}^{+\infty} \sum_{|\mathrm{k}|=k} \mathcal{P}_{\mathrm{k}}(\bx) a_{\mathrm{k}}
+\sum_{k=0}^{+\infty} \sum_{|\mathrm{k}|=k} \mathcal{Q}_{\mathrm{k}}(\bx) b_{\mathrm{k}},$$
where both the series converge uniformly on each compact subset of $\mathsf{A}$ and
$$a_{\mathrm{k}}=   \int_{\partial S_{\underline{\eta}}} \mathcal{Q}_{ \underline{\eta},\mathrm{k}}(\by) n(\by)f(\by) dS(\by), $$
$$b_{\mathrm{k}}= \int_{\partial S_{\underline{\eta}}} \mathcal{P}^{R}_{ \underline{\eta},\mathrm{k}}(\by) n(\by)f(\by) dS(\by), $$
  where  $n(\by)$ and $dS(\by)$  are given as in Theorem \ref{Cauchy-formula-slice}.
\end{theorem}
\begin{proof}
The result can be proved by using by the same strategy as in the proof of  Theorem \ref{Taylor-theorem}, so we omit the details.
\end{proof}

We conclude this section, by proving a maximum modulus principle for  generalized partial-slice monogenic functions which follows from  the Cauchy  formula in Theorem \ref{Cauchy-formula-monogenic}. For some related versions of the maximum modulus principle, we refer to \cite[Theorem 7.1]{Gentili-Stoppato-Struppa-13} for quaternionic  slice regular functions and \cite[Theorem 3.1]{Ren-Xu} for  slice monogenic functions.
\begin{theorem}{\bf (Maximum modulus principle)}
Let $\Omega\subseteq \mathbb{R}^{p+q+1}$ be a  slice domain and $f:\Omega\rightarrow \mathbb{R}_{p+q}$ be a  generalized partial-slice monogenic function.  If $|f|$ has a relative maximum at some point in $\Omega$, then $f$ is constant.
\end{theorem}
\begin{proof}
Let us assume that $|f|$ has a relative maximum at   $\bx \in \Omega$. We can write  $\bx=\bx_p+r  \underline{\omega}$ for some $\underline{\omega}\in \mathbb{S}$.  Let us consider the ball $B=B(\bx,\rho)=\{\by \in \mathbb{R}^{p+q+1}: |\by-\bx|<\rho \}$, and let $ \rho>0$ be small enough such that  $B_{\underline{{\omega}} }\subset \Omega_{\underline{{\omega}}} $. Then by Theorem \ref{Cauchy-formula-monogenic} we have
$$f(\bx)=\int_{\partial B_{\underline{\omega}} } E(\by- \bx) n(\by)f(\by) dS(\by), \quad  \bx \in B_{\underline{{\omega}}},   $$
where  $n(\by)=\frac{ \by-\bx}{|\by-\bx|}$ is the unit exterior normal to $\partial B_{\underline{\omega}}$ at $\by$.

Note that
$$E(\by- \bx ) n(\by)= \frac{1}{\sigma_{p+1}} \frac{\overline{\by-\bx}}{|\by-\bx|^{p+2}}\frac{\by-\bx}{|\by-\bx|}=
\frac{1}{\sigma_{p+1}} \frac{1}{|\by-\bx|^{p+1}}=\frac{1}{\sigma_{p+1}\rho^{p+1}}.$$
It follows that
 $$f(\bx)=\frac{1}{\sigma_{p+1}\rho^{p+1}}\int_{\partial B_{\underline{\omega}}} f(\by) dS(\by)=\frac{1}{\sigma_{p+1}\rho^{p+1}}\int_{\partial (B(0,\rho)_{\underline{\omega}})} f(\bx+\by) dS(\by),   $$
which implies that
$$|f(\bx)|\leq \frac{1}{\sigma_{p+1}\rho^{p+1}}\int_{\partial (B(0,\rho)_{\underline{\omega}})} |f(\bx+\by)| dS(\by)\leq|f(\bx)|.   $$
The above inequality forces the fact that  $|f_{\underline{\omega} }|$   is a constant in a small neighbourhood of $\bx$ in $\Omega_{\underline{\omega}}$, so is  $f_{\underline{\omega}}$.  To prove this assertion in precise terms, let us write
$$f_{\underline{\omega}}=\sum_{A\in \mathcal{P}(p+q)}f_{A}e_{A},  \quad f_{A}\in \mathbb{R},$$
where   $\mathcal{P}(p+q)$ is the permutation group with $p+q$ elements and $A$ is a multi-index.

Since $|f_{\underline{\omega}}|^{2}=\sum_{A\in \mathcal{P}(p+q)}f_{A}^{2}$ is constant, the derivatives of $|f_{\underline{\omega}}|^{2}$ with respect to variable $x_{i},i=0,1,\ldots,p+1,$ are zero, namely
$$\sum_{A\in \mathcal{P}(p+q)}f_A(\partial_{x_i}f_{A})=0.$$
A second differentiation with  respect to the variable $x_{i},i=0,1,\ldots,p+1,$  gives that
$$0=\sum_{A\in \mathcal{P}(p+q)} (\sum_{i=0}^{p+1}(\partial_{x_i}f_{A})^{2}+f_{A} (\Delta f_{A}))=\sum_{A\in \mathcal{P}(p+q)} (\sum_{i=0}^{p+1} \partial_{x_i}f_{A})^{2},$$
where $\Delta$ is the Laplacian   in $\mathbb{R}^{p+2}$.

Consequently, every $f_{A}$ is  constant in $\Omega_{\underline{\omega}}$, and so is  $f_{\underline{\omega}}$. Therefore,
 $f$ is constant in $\Omega$  by the identity theorem in Theorem \ref{Identity-theorem}.
\end{proof}

\section{Generalized partial-slice  functions}\label{Sec4}
The Representation Formula shows that generalized partial-slice monogenic functions have a special form, in fact they can be written as
$$
f(\bx_p+\underline{\omega}r)=F_1(\bx_p,r)+\underline{\omega}F_2(\bx_p,r)
$$
where $F_1$, $F_2$ do not depend on $\underline{\omega}$. When $p=0$ the function on the right-hand side is slice monogenic and of slice type. These functions appear in the Fueter-Sce construction, see \cite{Fueter, Colombo-Sabadini-Struppa-20}, as well as in the notion of slice regular functions as treated in \cite{Ghiloni-Perotti-11}. In this construction, one can use the notion of stem functions that we are generalizing to this case.

An open set $D$ of $\mathbb{R}^{p+2}$  is called invariant under the reflection  of the $(p+2)$-th variable if
$$ \bx':=(\bx_p,r) \in D \Longrightarrow   \bx_\diamond':=(\bx_p,-r)  \in D.$$
The  \textit{p-symmetric completion} $ \Omega_{D}$ of  $D$ is defined by
$$\Omega_{D}=\bigcup_{\underline{\omega} \in \mathbb{S}} \, \big \{\bx_p+r\underline{\omega}\  : \ \exists \ \bx_p \in \mathbb{R}_{p}^{0}\oplus \mathbb{R}_{p}^{1},\ \exists \ r\geq 0,\  \mathrm{s.t.} \ (\bx_p,r)\in D \big\}.$$
\begin{definition}
A function $F: D\longrightarrow  \mathbb{R}_{p+q} \otimes_{\mathbb R}\mathbb C$ in an open set $D\subseteq  \mathbb{R}^{p+2}$, which is invariant under the reflection  of the $(p+2)$-th variable, is called  a \textit{stem function} if
the $\mathbb{R}_{p+q} $-valued components  $F_1, F_2$ of $F=F_1+iF_2$ satisfy
$$ F_1(\bx_{\diamond}')= F_1(\bx'), \qquad  F_2(\bx_{\diamond}')=-F_2(\bx'), \qquad  \bx'=(\bx_p,r) \in D.$$
Each stem function $F$ induces a (left)  generalized partial-slice
function $f=\mathcal I(F): \Omega_{D} \longrightarrow \mathbb{R}_{p+q}$ given by
 $$f(\bx):=F_1(\bx')+\underline{\omega} F_2(\bx'), \qquad   \bx=\bx_p+r\underline{\omega}   \in \mathbb{R}^{p+q+1}, \underline{\omega}\in \mathbb{S}.$$
\end{definition}

We denote the set of  all induced  generalized partial-slice functions  on $\Omega_{D}$  by
 $$ {\mathcal{GS}}(\Omega_{D}):=\Big\{f=\mathcal I(F):    \ F \ {\mbox {is a stem function on }} D  \Big\}.$$
 Each generalized partial-slice function  $f$ is induced by a unique stem function $F$ since $F_1$ and $F_2$  are determined by  $f$. In fact, it holds that
$$F_1(\bx')=\frac{1}{2}\big( f(\bx)+f(\bx_{\diamond}) \big), \quad \bx'  \in D,$$
and
\begin{eqnarray*}
F_2(\bx')=
\left\{
\begin{array}{ll}
-\frac{1}{2  }\underline{\omega} \big( f(\bx)-f(\bx_{\diamond}) \big)     &\mathrm {if} \ \bx'\in  D\setminus \mathbb{R}^{p+1},
\\
0,   &\mathrm {if} \ \bx'\in D \cap \mathbb{R}^{p+1}.
\end{array}
\right.
\end{eqnarray*}

Furthermore,  we can establish the following formula  for generalized partial-slice functions.
\begin{theorem}  {\bf(Representation Formula, II)}  \label{Representation-Formula-SR}
Let $f\in {\mathcal{GS}}(\Omega_{D})$.  Then it holds that, for every  $\bx=\bx_p+r\underline{\omega} \in \Omega_{D}$ with $\underline{\omega}\in \mathbb{S}$,
\begin{equation*}\label{Rf of slice}
f(\bx)=(\underline{\omega}-\underline{\omega}_{2})(\underline{\omega}_{1}-\underline{\omega}_{2})^{-1}f(\bx_p+r\underline{\omega}_{1}) -(\underline{\omega}-\underline{\omega}_{1})(\underline{\omega}_{1}-\underline{\omega}_{2})^{-1}f(\bx_p+r\underline{\omega}_{2}),
\end{equation*}
 for all $\underline{\omega}_{1}\neq\underline{\omega}_{2}\in \mathbb{S}$.
 In particular, $\underline{\omega}_{1}=-\underline{\omega}_{2}=\underline{\eta}\in \mathbb{S},$
\begin{eqnarray*}
 f(\bx)
&=&\frac{ 1}{2} (1-\underline{\omega}\underline{\eta})    f(\bx_p+r\underline{\eta})+
 \frac{ 1}{2} (1+\underline{\omega}\underline{\eta})    f(\bx_p-r\underline{\eta})
 \\
 &=& \frac{1}{2} (f(\bx_p+r\underline{\eta} )+f(\bx_p-r\underline{\eta} ) )+
\frac{ 1}{2} \underline{\omega}\underline{\eta}(  f(\bx_p-r\underline{\eta} )-f(\bx_p+r\underline{\eta})).
\end{eqnarray*}
\end{theorem}
\begin{proof}
Let $\bx=\bx_p+r\underline{\omega} \in \Omega_{D}$ with $\underline{\omega}\in \mathbb{S}$.  By definition, it follows that, for all  $\underline{\omega}_{1},\underline{\omega}_{2}\in \mathbb{S}$,
 $$f(\bx_p+r\underline{\omega}_{1})=F_1(\bx')+\underline{\omega}_{1} F_2(\bx'),$$
and
$$f(\bx_p+r\underline{\omega}_{2})=F_1(\bx')+\underline{\omega}_{2} F_2(\bx').$$
Hence, for $\underline{\omega}_{1}\neq\underline{\omega}_{2},$
 $$ F_2(\bx')=(\underline{\omega}_{1}-\underline{\omega}_{2})^{-1}(f(\bx_p+r\underline{\omega}_{1})-f(\bx_p+r\underline{\omega}_{2})),$$
and then
  \begin{eqnarray*}
 F_1(\bx')
&=&f(\bx_p+r\underline{\omega}_{2})-\underline{\omega}_{2}F_2(\bx')
 \\
 &=&f(\bx_p+r\underline{\omega}_{2})-\underline{\omega}_{2}(\underline{\omega}_{1}-\underline{\omega}_{2})^{-1}
 (f(\bx_p+r\underline{\omega}_{1})-f(\bx_p+r\underline{\omega}_{2}))
 \\
 &=&\underline{\omega}_{1}(\underline{\omega}_{1}-\underline{\omega}_{2})^{-1}f(\bx_p+r\underline{\omega}_{2})-\underline{\omega}_{2}(\underline{\omega}_{1}-\underline{\omega}_{2})^{-1}
 f(\bx_p+r\underline{\omega}_{1}).
 \end{eqnarray*}
Therefore
  \begin{eqnarray*}
f(\bx)
&=&F_1(\bx')+\underline{\omega}F_2(\bx')
 \\
 &=&\underline{\omega}_{1}(\underline{\omega}_{1}-\underline{\omega}_{2})^{-1}f(\bx_p+r\underline{\omega}_{2})-\underline{\omega}_{2}(\underline{\omega}_{1}-\underline{\omega}_{2})^{-1}
 f(\bx_p+r\underline{\omega}_{1})
 \\
 &+&\underline{\omega}(\underline{\omega}_{1}-\underline{\omega}_{2})^{-1}(f(\bx_p+r\underline{\omega}_{1})-f(\bx_p+r\underline{\omega}_{2}))
 \\
 &=&(\underline{\omega}-\underline{\omega}_{2})(\underline{\omega}_{1}-\underline{\omega}_{2})^{-1}f(\bx_p+r\underline{\omega}_{1}) -(\underline{\omega}-\underline{\omega}_{1})(\underline{\omega}_{1}-\underline{\omega}_{2})^{-1}f(\bx_p+r\underline{\omega}_{2}),
\end{eqnarray*}
as desired.\end{proof}

The Representation Formula above  allows  to present  a Cauchy-Pompeiu integral  formula for generalized partial-slice functions. This formula is valid, in particular, on a p-symmetric slice domain, and so it is a further generalization of the Cauchy  formula in Theorem \ref{Cauchy-formula-slice}.

\begin{theorem}[Cauchy-Pompeiu  formula]\label{Cauchy-Pompeiu}
Let $f=\mathcal I(F)\in {\mathcal{GS}}(\Omega_{D})$ with  its stem  function $F\in C^{1}(\overline{D})$ and set $\Omega=\Omega_{D}$. If  $U$ is a domain in  $\mathbb{R}^{p+q+1}$ such that  $U_{\underline{\eta} }\subset \Omega_{\underline{\eta}} $  is a bounded domain in $\mathbb{R}^{p+2}$ with    smooth boundary $\partial U_{\underline{\eta}}\subset\Omega_{\underline{\eta}}$ for some $\underline{\eta}\in \mathbb{S}$, then
$$f(\bx)=\int_{\partial U_{\underline{\eta}}} \mathcal{E}_{\by}(\bx) n(\by)f(\by) dS(\by)-
\int_{  U_{\underline{\eta}}} \mathcal{E}_{\by}(\bx)  (D_{\underline{\eta}}f)(\by) dV(\by), \quad  \bx \in U,   $$
where  $n(\by)=\sum_{i=0}^{p}n_i(\by) e_i+n_{p+1}(\by)\underline{\eta}$ is the unit exterior normal to $\partial U_{\underline{\eta}}$ at $\by$,
 $dS$ and  $dV$ stand  for  the   classical   Lebesgue surface  element and volume element  in $\mathbb{R}^{p+2}$, respectively.
\end{theorem}

\begin{proof}
Let $\bx=\bx_p+r \underline{\omega}  \in U, \bx'=(\bx_p,r)\in D$ and $\underline{\eta} \in \mathbb{S}$. Set
$$f(\bx)=F_1(\bx')+ \underline{\omega}  F_2(\bx'),$$
where the stem  function $F=F_1+iF_2\in C^{1}(D)$.  In fact, there exist  $C^{1}$ functions $F_{A}^{j}:D \rightarrow \mathbb{R}_{p+1}$ $(j=1,2)$ such that
$$ F^j = \sum_{A=\{i_{1},\ldots, i_{s}\} \subseteq \{p+2,\ldots,p+q\}}  F_{A}^{j} e_{A}, \quad j=1,2.$$
where $e_{A}=e_{i_{1}}\cdots e_{i_{s}},  A=\{i_{1},\ldots, i_{s}\} \subseteq \{p+2,\ldots,p+q\}$ with $i_{1}<\cdots<i_{s}$, and $I_{\emptyset}=1$ when $A=\emptyset$.

The   Cauchy-Pompeiu formula  (see e.g. \cite[Theorem 9.5]{Brackx}) for $C^{1}$ functions gives that
$$F_{A}^{j}(\bx')=\int_{\partial U_{\underline{\eta}}} E_{\by}(\bx_p+r \underline{\eta} ) n(\by) F_{A}^{j} (\by) dS(\by)-
\int_{U_{\underline{\eta}}} E_{\by}(\bx_p+r \underline{\eta})  (D_{\underline{\eta}}F_{A}^{j})(\by) dV(\by),    $$
where  $n, dS, dV$  are as claimed in the theorem and $D_{\underline{\eta}}F_{A}^{j}=(\sum_{i=0}^{p}e_i\partial_{x_i}+\underline{\eta} \partial_{r})F_{A}^{j}$.
Here the variable  $\by=\by_p+\widetilde{r} \underline{\eta} \in \partial U_{\underline{\eta}}$  should be  interpreted as $(\by_p,\widetilde{r} )\in \mathbb{R}^{p+2}$.

Hence,
$$F(\bx')=\int_{\partial U_{\underline{\eta}}} E_{\by}(\bx_p+r \underline{\eta} ) n(\by) F(\by) dS(\by)-
\int_{U_{\underline{\eta}}} E_{\by}(\bx_p+r \underline{\eta} )  (D_{\underline{\eta}}F)(\by) dV(\by).    $$
By replacing $i$ with $\underline{\eta}$ for the stem function $F$ in the above  formula, we  obtain
$$f(\bx_p+r\underline{\eta})=\int_{\partial U_{\underline{\eta}}} E_{\by}(\bx_p+r \underline{\eta}) n(\by) f(\by) dS(\by)-
\int_{U_{\underline{\eta}}} E_{\by}(\bx_p+r\underline{\eta})  (D_{\underline{\eta}}f)(\by) dV(\by).    $$
Furthermore, the formula above still holds when  the variable $\bx_p+r \underline{\eta}$ is replaced by $\bx_p-r \underline{\eta}$.
By using  the Representation Formula  in Theorem \ref{Representation-Formula-SR}, it follows that
$$
 f(\bx_p+r \underline{\omega} )=\frac{ 1}{2} (1- \underline{\omega}\underline{\eta})    f(\bx_p+r\underline{\eta})+
 \frac{ 1}{2} (1+ \underline{\omega}\underline{\eta})    f(\bx_p-r\underline{\eta}),$$
from which we  obtain the   conclusion
$$f(\bx)=\int_{\partial U_{\underline{\eta}}} \mathcal{E}_{\by}(\bx) n(\by)f(\by) dS(\by)-
\int_{  U_{\underline{\eta}}} \mathcal{E}_{\by}(\bx)  (D_{\underline{\eta}}f)(\by) dV(\by), \quad  \bx \in U.   $$
The proof is complete.

\end{proof}

Let us set
 $${\mathcal{GS}}^{j}(\Omega_{D}):=\Big\{f=\mathcal I(F):    \ F \ {\mbox {is a}}\  C^{j} \ {\mbox {stem function on }} D  \Big\}, \quad j=0,1.$$

\begin{definition}\label{definition-GSR}
Let $f \in {\mathcal{GS}}^{1}(\Omega_{D})$. The function $f$ is called generalized partial-slice monogenic  of type $(p,q)$ if its stem function $F=F_1+iF_2$ satisfies  the generalized Cauchy-Riemann equations
 \begin{eqnarray}\label{C-R}
 \left\{
\begin{array}{ll}
D_{\bx_p}  F_1- \partial_{r} F_2=0,
\\
 \overline{D}_{\bx_p}  F_2+ \partial_{r} F_1=0.
\end{array}
\right.
\end{eqnarray}
\end{definition}
We denote by $\mathcal {GSR}(\Omega_D)$ the set of all generalized partial-slice monogenic functions {on $\Omega_D$}.  We note that if confusion may arise with the class of functions introduced earlier in Definition   \ref{definition-slice-monogenic}, one can call them generalized partial-slice regular, following \cite{Ghiloni-Perotti-11}.
As before, the type $(p,q)$ will be omitted in the sequel.

Now we present a relationship between the set of functions  $\mathcal {GSM}$ and $\mathcal {GSR}$ defined in  p-symmetric domains.

\begin{theorem} \label{relation-GSR-GSM}

(i) For a p-symmetric domain $\Omega=\Omega_{D}$ with $\Omega  \cap \mathbb{R}^{p+1}= \emptyset$, it holds that $\mathcal {GSM}(\Omega) \supsetneqq \mathcal {GSR}(\Omega_{D})$.

(ii) For a p-symmetric domain $\Omega=\Omega_{D}$ with $\Omega  \cap \mathbb{R}^{p+1}\neq \emptyset$,  it holds  that $\mathcal {GSM}(\Omega) = \mathcal {GSR}(\Omega_{D})$.
\end{theorem}

 \begin{proof}
$(i)$  Let $f=\mathcal I(F)\in  \mathcal {GSR}(\Omega_{D})$ with   its stem  function $F\in C^{1}(D)$ satisfying  the generalized Cauchy-Riemann equations
 \begin{eqnarray*}
 \left\{
\begin{array}{ll}
D_{\bx_p}  F_1- \partial_{r} F_2=0,
\\
 \overline{D}_{\bx_p}  F_2+ \partial_{r} F_1=0.
\end{array}
\right.
\end{eqnarray*}
 Then
 \begin{eqnarray*}
 D_{\underline{\omega}}f(\bx) &=& (D_{\bx_p}+\underline{\omega}\partial_{r})(F_1(\bx')+\underline{\omega} F_2(\bx'))
 \\
 &=& D_{\bx_p}  F_1(\bx')- \partial_{r} F_2(\bx')+ \underline{\omega} (\overline{D}_{\bx_p}  F_2(\bx')+ \partial_{r} F_1(\bx'))=0,
\end{eqnarray*}
which means that $f\in\mathcal {GSM}(\Omega)$.

To see that the inclusion is strict, consider the function
$$f(\bx) = \underline{\omega}, \quad \bx \in\Omega=\mathbb{R}^{p+q+1}\setminus \mathbb{R}^{p+1},$$
where $\bx=\bx_p+r\underline{\omega}$ with   $\bx_p\in \mathbb{R}^{p+1}, r>0, \underline{\omega}\in \mathbb{S}$.

It is immediate that $f\in  \mathcal {GSM}(\Omega)$ but $f\notin\mathcal {GSR}(\Omega)$.

$(ii)$  From the proof of   (i), we have the inclusion  $\mathcal {GSM}(\Omega) \supseteq \mathcal {GSR}(\Omega_{D})$ for $\Omega=\Omega_{D}$,  It remains to show   $\mathcal {GSM}(\Omega) \subseteq \mathcal {GSR}(\Omega_{D})$. Let $f\in  \mathcal {GSM}(\Omega)$.  Note that $\Omega_{D}$ is p-symmetric. If $\Omega=\Omega_{D}$ is a slice domain, then the Representation Formula in Theorem \ref{Representation-Formula-SM} holds
$$f(\bx)=f(\bx_p+r \underline{\omega})=  F_1(\bx_p,r)+\underline{\omega}F_2(\bx_p,r),$$
where $F_1, F_2$ are defined as is  Theorem \ref{Representation-Formula-SM}.
In fact,  $ F_1+i F_2$ is a stem function and satisfies  the generalized Cauchy-Riemann equations:
 \begin{eqnarray*}
D_{\bx_p}  F_1 (\bx_p,r)&=& \frac{1}{2} (D_{\bx_p} f(\bx_p+r\underline{\omega} )+D_{\bx_p}f(\bx_p-r\underline{\omega} )  )
\\
 &=&  \frac{1}{2} (- \underline{\omega}\partial_{r}  f(\bx_p+r\underline{\omega}) +\underline{\omega}\partial_{r}  f(\bx_p-r\underline{\omega})  )
 \\
&=& \partial_{r} F_2 (\bx_p,r),
\end{eqnarray*}
similarly, \begin{eqnarray*}
\overline{D}_{\bx_p}  F_2 (\bx_p,r)&=&\frac{1}{2} \underline{\omega}D_{\bx_p}( f(\bx_p-r\underline{\omega} )- f(\bx_p+r\underline{\omega} )  )
\\
 &=&\frac{1}{2} \underline{\omega} (- \underline{\omega}\partial_{r} f(\bx_p-r\underline{\omega} )-  \underline{\omega}\partial_{r}  f(\bx_p+r\underline{\omega})  )
 \\
 &=&\frac{1}{2}  \partial_{r} (   f(\bx_p-r\underline{\omega} )+    f(\bx_p+r\underline{\omega})  )
 \\
&=& \partial_{r} F_1(\bx_p,r).
\end{eqnarray*}
 Hence, $f \in\mathcal {GSR}(\Omega)$, as stated.
\end{proof}

\section{Global differential operators}

The definition of generalized partial-slice monogenic function is based on the validity of an infinite number of generalized Cauchy-Riemann equations, parametrized by $\underline{\omega}\in\mathbb S$. One may wonder if, like in the slice monogenic case, one may find a global operator associated to this notion. The answer is affirmative and is discussed in this section.

Let us define the nonconstant coefficients differential operator
 $$G_{\boldsymbol{x}}f(\bx)=|\underline{\bx}_q|^2\Dxp f(\bx)+ \underline{\bx}_q \mathbb{E}_{\underline{\boldsymbol{x}}_{q}} f(\bx), $$
for $C^{1}$ functions $f :\Omega \rightarrow {\mathbb{R}_{p+q}}$, where $\Omega$ is a domain in $\mathbb{R}^{p+q+1}$.

Note that the global differential operators $G_{\bx}$ was firstly introduced  in 2013, see \cite{Colombo-Cervantes-Sabadini-13}, for  slice monogenic functions that is in the case $p=0$. The operator $G_{\bx}$  is not elliptic and there is a degeneracy in $\underline{\bx}_q=0$.  Its kernel
contains also less smooth functions that have to be interpreted as distributions \cite{Colombo-Sommen}.

Let $\Omega$ be a domain in $\mathbb{R}^{p+q+1}\setminus \mathbb{R}^{p+1}$.  We can define another related  global  nonconstant coefficients differential operator for $C^{1}$ function $f :\Omega \rightarrow \mathbb{R}_{p+q}$:
 $$\overline{\vartheta} f(\boldsymbol{x})=\Dxp f(\bx)+\frac{\underline{\bx}_q}{|\underline{\bx}_q|^2}\mathbb{E}_{\underline{\boldsymbol{x}}_{q}} f(\bx).$$
This operator  for $p=0$ was  firstly introduced  in the case of slice regularity on a real alternative algebra \cite{Ghiloni-Perotti-14}.

By definition,  a direct observation gives following results.

\begin{proposition}\label{relation-operator}
Let $\Omega$ be a domain in $\mathbb{R}^{p+q+1}\setminus \mathbb{R}^{p+1}$. For the $C^{1}$ function $f :\Omega \rightarrow {\mathbb{R}_{p+q}}$, it holds that

(i) {$\overline{\vartheta} f(\bx)=D_{\underline{\omega}}f(\bx),\bx=\bx_{p}+r\underline{\omega}$;}

 (ii) $f\in ker \overline{\vartheta}  \Leftrightarrow  f\in ker G_{\bx} \Leftrightarrow f\in \mathcal {GSM}(\Omega). $
 \end{proposition}

To present  some  relations between the nonconstant coefficients differential operator  $\overline{\vartheta}$ with the generalized Cauchy-Riemann operator $D_{\bx}$, we need to introduce some more terminology. Below $\Omega_D$ is as in Section \ref{Sec4}.

\begin{definition}
Let $f\in   {\mathcal{GS}}(\Omega_{D}) $.
The function $f^{\circ}_{s}: \Omega_{D} \longrightarrow \mathbb{R}_{p+q}$, called spherical value of $f$, and the   function $f^{'}_{s}: \Omega_{D}\setminus \mathbb{R}^{p+1} \longrightarrow \mathbb{R}_{p+q}$  called spherical derivative of $f$, are defined as, respectively,
$$f^{\circ}_{s}(\bx)=\frac{1}{2}\big( f(\bx)+f(\bx_{\diamond}) \big), $$
$$f^{'}_{s}(\bx)=\frac{1}{2}\underline{\bx}_{q}^{-1}\big( f(\bx)-f(\bx_{\diamond}) \big).$$
\end{definition}

By definition, one can get easily two formulas that
 $f^{\circ}_{s}(\bx)=F_{1}(\bx')$ and $f^{'}_{s}(\bx)=r^{-1}F_{2}(\bx')$,  being constant on every set $\bx_{p}+r\mathbb{S}$.
Hence, $f^{\circ}_{s}$ and $f^{'}_{s}$ are  generalized partial-slice  functions. Furthermore, it holds that

\begin{equation}\label{spherical-value-derivative}
 f(\bx)=f^{\circ}_{s}(\bx)+ \underline{\bx}_{q} f_{s}^{'}(\bx), \quad \bx \in \Omega_{D}\setminus \mathbb{R}^{p+1}.
\end{equation}


For any $i,j$ with  $p+1\leq i<j\leq p+q$, the spherical Dirac operator on $\mathbb{R}^{q}$ is given by
$$\Gamma=-\sum_{i<j} e_{i}e_{j}L_{ij},$$
where
$ L_{ij}=x_{i} \frac{\partial}{\partial x_{j}}-x_{j} \frac{\partial}{\partial x_{i}}$ {are} the angular momentum operators.

With   spherical Dirac operator,   Dirac operator $D_{\underline{\bx}_q}$ has a decomposition (see e.g., \cite[Section 0.1.4]{Delanghe-Sommen-Soucek})
$$D_{\underline{\bx}_q}=\underline{\omega}\partial_{r}+\frac{1}{r}\underline{\omega}\Gamma,\quad \underline{\bx}_q=r\underline{\omega}, \partial_{r}=\frac{1}{r}\E_{\underline{\bx}_q}.$$
Hence,
 \begin{equation}\label{Dirac-operator-decomp}
 D_{\bx}=D_{\bx_p}+\underline{\omega}\partial_{r}+\frac{1}{r}\underline{\omega}\Gamma,\end{equation}
where $\bx=\bx_p+\underline{\bx}_q=\bx_p+r\underline{\omega}.$

 With the decomposition of generalized partial-slice functions (\ref{spherical-value-derivative}), we have
\begin{equation}\label{spherical-Dirac-operator}
\Gamma f(\bx)= (\Gamma \underline{\bx}_{q}) f_{s}^{'}(\bx)=(q-1)\underline{\bx}_{q}f_{s}^{'}(\bx), \quad \bx \in \Omega_{D}\setminus \mathbb{R}^{p+1}.
\end{equation}

From   (\ref{Dirac-operator-decomp}) and (\ref{spherical-Dirac-operator}), it follows that
$$(D_{\bx}-{D_{\underline{\omega}}} )f(\bx) =\frac{1}{r}\underline{\omega}\Gamma f(\bx)=(1-q)f_{s}^{'}(\bx). $$

Summarizing the discussion above,  we obtain  the following property. See \cite[Proposition 3.3.2]{Perotti-19} for the case of slice regular functions.
\begin{proposition}\label{relation}
If $f\in {\mathcal{GS}}^{1}(\Omega_{D}) $, then the following two formulas hold on $\Omega_{D}\setminus \mathbb{R}^{p+1}$:

(i) $\Gamma f(\bx)=(q-1)\underline{\bx}_{q}f_{s}^{'}(\bx)$;

 (ii) $(D_{\bx}-{D_{\underline{\omega}})}f(\bx) =(D_{\bx}- \overline{\vartheta} )f(\bx)=(1-q)f_{s}^{'}(\bx). $
 \end{proposition}

   When the stem  function $F\in C^{1}(D)$, $f_{s}^{'}(\bx)$ could extend continuously to  $\Omega_{D}\cap \mathbb{R}^{p+1}$ with   values  being $\partial_{r}F_{2}(\bx_{p},0)$.  At this point, (\ref{spherical-value-derivative}) holds naturally in the whole domain  $\Omega_{D}$. Therefore, based on Proposition \ref{relation},  we  redefine the nonconstant coefficients global differential operator $\overline{\vartheta}:\ {\mathcal{GS}}^{1}(\Omega_{D})\cap C^{1}(\Omega_{D})\rightarrow {\mathcal{GS}}^{0}(\Omega_{D})$ as
  \begin{eqnarray*}
 \overline{\vartheta} f(\bx)=\left\{
\begin{array}{ll}
\Dxp f(\bx)+\frac{\underline{\bx}_q}{|\underline{\bx}_q|^2}\mathbb{E}_{\underline{\boldsymbol{x}}_{q}}f(\bx), \quad  \bx \in \Omega_{D}\setminus \mathbb{R}^{p+1},
\\
  D_{\bx}f(\bx)+(q-1)f_{s}^{'}(\bx), \quad  \bx \in \Omega_{D} \cap \mathbb{R}^{p+1}.
\end{array}
\right.
\end{eqnarray*}

Hence, we obtain   relationships from Theorem  \ref{relation-GSR-GSM} and Proposition  \ref{relation-operator}.
\begin{theorem}\label{relationship-g}

(i) For the p-symmetric domain $\Omega$ with $\Omega  \cap \mathbb{R}^{p+1}= \emptyset$ and $f\in C^{1}(\Omega)$, it holds that
   $$f\in ker \overline{\vartheta}  \Leftrightarrow  f\in ker G_{\bx} \Leftrightarrow f\in \mathcal {GSM}(\Omega). $$

(ii)For the p-symmetric domain $\Omega=\Omega_{D}$ with $\Omega  \cap \mathbb{R}^{p+1}\neq \emptyset$ and $f\in C^{1}(\Omega)\cap \mathcal{GS}^{1}(\Omega_{D})$,  it holds  that
   $$f\in ker \overline{\vartheta}  \Leftrightarrow  f\in \mathcal {GSM}(\Omega) \Leftrightarrow f\in \mathcal {GSR}(\Omega_{D}). $$

\end{theorem}

\section{Conformal invariance}

To offer a rather complete discussion of the topic of generalized partial-slice monogenic functions, in this section we shall discuss their conformal invariance. We first recall a classical result for monogenic functions, see \cite{Ryan}.
\begin{theorem}\label{Ryan}
Let $M =\begin{pmatrix} a \ b \\ c\ d\end{pmatrix} \in  GAV(\mathbb{R}\oplus \mathbb{R}^{n})$  and  $f$ be a left monogenic function in a domain $\Omega\subseteq \mathbb{R}^{n+1} $. Then   the function
\begin{eqnarray*}\label{inter}
 \frac{\overline{cx+d}}{|cx+d|^{n+1}}f((ax+b)(cx+d)^{-1})
\end{eqnarray*}
is also left monogenic in the variable $x\in M^{-1}(\Omega)$.
\end{theorem}
In the context of slice monogenic functions, one has to restrict the class of the M\"obius transformations preserving the axial symmetry, see \cite{Colombo-Krau-Sabadini-20}. Also in the case of generalized partial-slice monogenic functions,  we need to define a subgroup of M\"obius group  preserving the p-symmetry, following \cite{Colombo-Krau-Sabadini-20}. We consider:
\begin{eqnarray*}\label{GRAV}
GRAV( \mathbb{R}\oplus \mathbb{R}^{p+q})=
\left\langle
\bpm
1& b\\0&1
\epm,
\bpm
a&0\\ 0&a^{-1}
\epm,
\bpm
0&1\\-1&0
\epm,
\bpm
\lambda&0\\0&\lambda^{-1}
\epm
\right\rangle,
\end{eqnarray*}
where $b\in \mathbb{R}^{p+1},\ a\in \mathbb{S}$ and $\lambda\in \mathbb{R}\backslash\{0\}$.

When $p=0,q=n$, the group $GRAV( \mathbb{R}\oplus \mathbb{R}^{n})$ is that one in \cite[Definition 2.12]{Colombo-Krau-Sabadini-20}.  In our more general setting, we now prove that the group $GRAV( \mathbb{R}\oplus \mathbb{R}^{p+q})$  preserves  the property of sliceness and p-symmetry.
\begin{proposition}\label{symmetry}
 The elements in the group $GRAV( \mathbb{R}\oplus \mathbb{R}^{p+q})$ take  p-symmetric
slice domains into p-symmetric slice domains.
\end{proposition}

 \begin{proof}
Let $\Omega$ be a p-symmetric slice domain, $\bx=\bx_p+r\underline{\omega}\in \Omega$ with   $\bx_p\in \mathbb{R}^{p+1}, r\geq0, \underline{\omega}\in \mathbb{S}$. Let $M$ be one of the generators of the group $GRAV( \mathbb{R}\oplus \mathbb{R}^{p+q})$: we consider
 the four types of generators.

(i) Translation: $M=\begin{pmatrix} 1 \ \ b \\ 0 \ \ 1\end{pmatrix}$ with $b\in \mathbb{R}^{p+1}$; the induced function $M\langle \bx\rangle $ maps  $\bx$ to
$$(\bx_p+b)+ r\underline{\omega} \in  M(\Omega). $$
Obviously, $M(\Omega)$ is also a p-symmetric slice domain.

(ii) Inversion: $M=\begin{pmatrix}  0 \ \ \ 1\\ -1\ \ 0\end{pmatrix}$; the induced function takes the form
$$ M\langle \bx\rangle =-\bx^{-1}=-\frac{\overline{\bx_{p}}} {|\bx|^{2}}+ \frac{r}{|\bx|^{2}}\underline{\omega}, $$
which implies that $M(\Omega)$ is  a p-symmetric slice domain.

(iii) Modified rotation: $\begin{pmatrix}  a\ \ \ \ 0  \\ 0 \ \  \ a^{-1}\end{pmatrix}$ with $ a\in \mathbb{S},$
$$M\langle \bx\rangle =a \bx a =a\bx_pa+ra\underline{\omega}a=-\overline{\bx_p}+ra\underline{\omega}a.$$
Note that $(a\underline{\omega}a)^{2}=-1,$ and $a\underline{\omega}a=-2[a\overline{\underline{\omega}}]_{0}a+\underline{\omega}\in \mathbb{R}^{q},$
which implies that $a\underline{\omega}a\in \mathbb{S}.$ Furthermore, the map $\Psi:\mathbb{S}\rightarrow\mathbb{S}, \underline{\omega} \rightarrow a\underline{\omega}a$ is a covering. Hence,  $M(\Omega)$ is    a p-symmetric slice domain.

(iv) Dilation: $\begin{pmatrix}  \lambda\ \ \ \  \ 0  \\  0 \ \ \lambda^{-1} \end{pmatrix}$ with $\lambda \in \mathbb{R}\backslash\{0\},$
$$M\langle \bx\rangle =\lambda^{2}\bx=\lambda^{2}\bx_p+(\lambda^{2}r)\underline{\omega},$$
Obviously, $M(\Omega)$ is    a p-symmetric slice domain.
\end{proof}

By using of  Theorem  \ref{Representation-Formula-SM} and Proposition  \ref{symmetry}, we prove the invariance of generalized partial-slice monogenic functions under the group $GRAV( \mathbb{R}\oplus \mathbb{R}^{p+q})$   following \cite[Theorem 3.1]{Colombo-Krau-Sabadini-20}.
\begin{theorem}\label{invariance}
Let $M=\begin{pmatrix} a \ b \\ c\ d\end{pmatrix} \in GRAV(\mathbb{R}\oplus \mathbb{R}^{p+q})$  and let $f\in \mathcal {GSM}(\Omega)$, where   $\Omega\subseteq \mathbb{R}^{p+q+1}$ is a p-symmetric slice domain. Then the function
$$ T_{f}(\bx)=J(M,\bx)f(M\langle \bx\rangle), \quad J(M,\bx)=\frac{\overline{c\bx+d}}{|c\bx+d|^{p+2}},$$
is also  generalized partial-slice monogenic   in the variable $\bx \in M^{-1}(\Omega)$.
\end{theorem}

\begin{proof}
Let $\Omega$ be a p-symmetric slice domain. We prove the assertion for the generators $M$  of the group $GRAV( \mathbb{R}\oplus \mathbb{R}^{p+q})$.
Given $\bx=\bx_p+r\underline{\omega}\in \Omega$ with   $\bx_p\in \mathbb{R}^{p+1}, r\geq0, \underline{\omega}\in \mathbb{S}$, we consider
 $M(\bx)$ in four cases.

(i) Translation: $M=\begin{pmatrix} 1 \ \ b \\ 0 \ \ 1\end{pmatrix}$   with $b\in \mathbb{R}^{p+1}$. Then it is immediate that
$$T_{f}(\bx)=f(M\langle \bx\rangle)=f(\bx+b) $$
is still    generalized partial-slice monogenic   in $ M^{-1}(\Omega)$.

(ii) Inversion: $M=\begin{pmatrix}  0 \ \ \ 1\\ -1\ \ 0 \end{pmatrix}$. Then
$$T_{f}(\bx)=\frac{\overline{-\bx }}{|- \bx|^{p+2}}f(-\bx^{-1}). $$
If we identify $\bx$ with $\bx'=(\bx_{p},r)$, then   by  Theorem \ref{Ryan} in the case $n=p+1$, we get that the restriction  of the function $T_{f}$ to $M^{-1}(\Omega)_{\underline{\omega}}$ is left monogenic  for each $\underline{\omega}\in \mathbb{S}$. Consequently,  $T_{f} \in
\mathcal{GSM}(M^{-1}(\Omega))$.

(iii) Modified rotation: $M=\begin{pmatrix}  a\ \ \ \ \ \  \\  \ \  \ a^{-1}\end{pmatrix}$ with $ a\in \mathbb{S}.$ From the proof of Proposition \ref{symmetry},  it holds that
$$M\langle \bx\rangle =a \bx a  =-\overline{\bx_p}+ra\underline{\omega}a \in \mathbb{R}^{p+1} \oplus r \mathbb{S}.$$
Thus by setting $a\underline{\omega}a =\underline{\omega}'$ and using the Representation Formula
 in Theorem \ref{Representation-Formula-SM}, we deduce that the  generalized partial-slice monogenic function $f$ defined over the p-symmetric slice domain $\Omega$ has the form
 $$f(a\bx a)=  F_1+\underline{\omega}'F_2,$$
where $F_1,F_2$ depend on suitable variables.\\
More precisely, we have
$$T_{f}(\bx)=  a f(a \bx a)
= a( F_1(-\overline{\bx_p},r)+a\underline{\omega}aF_2(-\overline{\bx_p},r))
 = a F_1(-\overline{\bx_p},r)-\underline{\omega}aF_2(-\overline{\bx_p},r).$$
From the following the equalities
$$ D_{\bx_p}  \underline{\omega}a= \underline{\omega}\overline{D}_{\bx_p} a=\underline{\omega} a D_{\bx_p},$$
$$\overline{D}_{\bx_p} ( F_1(-\overline{\bx_p},r))=- (D_{\bx_p} F_1 )  (-\overline{\bx_p},r ),$$
$$ D_{\bx_p} (F_2(-\overline{\bx_p},r))=-(\overline{D}_{\bx_p}F_2 )(-\overline{\bx_p},r), $$
we obtain that
\begin{eqnarray*}
& & D_{\underline{\omega}} T_{f}(\bx)\\
&=&(D_{\bx_p}+\underline{\omega}\partial_{r})(a F_1(-\overline{\bx_p},r)-\underline{\omega}a F_2(-\overline{\bx_p},r))
\\
 &=& D_{\bx_p} (a F_1(-\overline{\bx_p},r))- D_{\bx_p} (\underline{\omega}a F_2(-\overline{\bx_p},r))+\underline{\omega}\partial_{r}(a F_1(-\overline{\bx_p},r)-\underline{\omega}aF_2(-\overline{\bx_p},r))
 \\
 &=& a \overline{D}_{\bx_p} ( F_1(-\overline{\bx_p},r))- \underline{\omega}a D_{\bx_p} (F_2(-\overline{\bx_p},r))+\underline{\omega}a \partial_{r} F_1(-\overline{\bx_p},r)+a \partial_{r} F_2(-\overline{\bx_p},r)
 \\
&=&  -a (D_{\bx_p} F_1)  (-\overline{\bx_p},r )+  \underline{\omega}a (\overline{D}_{\bx_p}F_2 )(-\overline{\bx_p},r))+\underline{\omega}a \partial_{r} F_1(-\overline{\bx_p},r)+a \partial_{r} F_2(-\overline{\bx_p},r)
 \\
&=&  -a (D_{\bx_p} F_1-\partial_{r} F_2)  (-\overline{\bx_p},r )+
\underline{\omega}a (\overline{D}_{\bx_p}F_2+ \partial_{r} F_1)(-\overline{\bx_p},r).
\end{eqnarray*}
By Theorem \ref{relation-GSR-GSM} (ii),  the condition $f \in
\mathcal{GSM}(\Omega )$ is equivalent to that the pair $(F_1,F_2)$ satisfies the generalized Cauchy-Riemann equations (\ref{C-R}), which implies that $D_{\underline{\omega}} T_{f}(\bx)\equiv0$ in $M^{-1}(\Omega)$, i.e.,  $T_{f} \in
\mathcal{GSM}(M^{-1}(\Omega))$.

(iv) Dilation: $M=\begin{pmatrix}  \lambda\ \ \ \  \ 0  \\ 0 \ \ \lambda^{-1} \end{pmatrix}$ with $\lambda \in \mathbb{R}\backslash\{0\}.$ Then
  $$T_{f}(\bx)=|\lambda|^{p+2} \lambda^{-1} f (\lambda^{2}\bx)$$
is  obviously generalized partial-slice monogenic   in $ M^{-1}(\Omega)$.

The final step is to prove the statement for any element of the group. In fact, we need only to consider two generators $M=\begin{pmatrix} a \ b \\ c\ d\end{pmatrix}, N=\begin{pmatrix} A \ B \\ C\ D\end{pmatrix} \in GRAV(\mathbb{R}\oplus \mathbb{R}^{p+q})$.

As we proved before, the condition  $f \in
\mathcal{GSM}(  \Omega )$ implies that $T_{f} \in \mathcal{GSM}(M^{-1}(\Omega)),$ which also leads to the function
$$   g(x)= J(N,\bx)T_{f}(N\langle \bx\rangle) \in \mathcal{GSM}((MN)^{-1}(\Omega)). $$
Direct computations give that
$$ J(M,N\langle \bx\rangle)= \frac{\overline{cN\langle \bx\rangle+d}}{|cN\langle \bx\rangle+d|^{p+2}}=
\frac{|C\bx+ D|^{p+2}}{\overline{  C\bx+ D  }} \frac{\overline{(cA+dC) \bx+cB+dD }}{|(cA+dC) \bx+cB+dD|^{p+2}},$$
and  then
\begin{eqnarray*}
 J(N,\bx)J(M,N\langle \bx\rangle)&=&\frac{\overline{C\bx+D}}{|C\bx+D|^{p+2}}J(M,N\langle \bx\rangle)
\\
 &=&  \frac{\overline{(cA+dC) \bx+cB+dD }}{|(cA+dC) \bx+cB+dD|^{p+2}}
 \\
 &=&  J((MN), \bx).
\end{eqnarray*}
 Therefore,
$$   g(x)= J(N,\bx)J(M,N\langle \bx\rangle)f(M\langle N\langle \bx\rangle  \rangle) =J((MN),\bx) f((M N)\langle \bx\rangle ). $$
Since the transformation $T_{f}$ preserves the generalized partial-slice monogenicity under  the composition of  any two generators $M$ and $N$, the desired result follows. \end{proof}

In the proof of Theorem \ref{invariance} above, the assumption of `p-symmetric slice domain' is fully applied when using  Theorem \ref{Representation-Formula-SM} and  Theorem \ref{relation-GSR-GSM} (ii).  From this point of view, it seems  that  Theorem \ref{invariance}  is valid  under the condition that the involved domain is partially symmetric and slice. However, the sliceness is not required in \cite[Theorem 3.1]{Colombo-Krau-Sabadini-20} since the authors use the class of slice functions which are also slice monogenic. In any case, an elementary approach  to deal with the   invariance of generalized partial-slice monogenicity and  establish  {an}  intertwining relation  for $C^{1}$ functions in term of  the global  differential operator   $G_{\bx}$ and its    iterated  case is developed in another  paper \cite{Ding}.
\bigskip
\\
\textbf{Acknowledgements}
The first author  would like express his hearty thanks  to Dr. Chao Ding (Anhui University) for helpful discussions.
\bigskip
\\
\textbf{Conflict of interest}
On behalf of all authors, the corresponding author states that there is no conflict of interest.




\vskip 10mm
\end{document}